\documentclass[letterpaper,11pt]{article}
\usepackage[margin=1in]{geometry}
\usepackage[bookmarks, plainpages = false, colorlinks=true,
            citecolor=red,
            linkcolor=blue,
            anchorcolor=red,
            urlcolor=blue,
            ]{hyperref}
\usepackage{url}
\usepackage{array}
\usepackage{amsmath,amsfonts,amsthm,amssymb,bm, verbatim,dsfont,mathtools}
\usepackage{color,graphicx,appendix}
\usepackage{subfigure}
\usepackage{etoolbox}
\usepackage{tikz}
\usepackage{multirow}
\usepackage{tabularray}
\usepackage{xr,xspace}
\usepackage{todonotes}
\usepackage{paralist}
\usepackage{caption,soul}
\usepackage{algorithm}
\usepackage{algorithmic}
\makeatletter

\theoremstyle{plain}
\newtheorem{theorem}{Theorem}
\newtheorem{lemma}{Lemma}

\newtheorem{corollary}{Corollary}
\theoremstyle{definition}
\newtheorem{definition}{Definition}

\newtheorem{remark}{Remark}
\newtheorem*{remark*}{Remark}
\newtheorem{example}{Example}

\usepackage{xspace,prettyref}
\newcommand{\Ichi}{I_{\chi^2}}
\newcommand{\Ichir}{I_{\chi^2}^{(r)}}
\newcommand{\diverge}{\to\infty}

\newcommand{\zeros}{\mathbf 0}
\newcommand{\reals}{{\mathbb{R}}}

\newcommand{\rexp}[1]{e^{#1}}
\newcommand{\identity}{\mathbf I}

\newcommand{\diff}{{\rm d}}

\newcommand{\Expect}{\mathbb{E}}

\newcommand{\Prob}{\mathbb{P}}

\newcommand{\Bern}{{\rm Bern}}

\newcommand{\ie}{i.e.\xspace}

\newrefformat{eq}{(\ref{#1})}
\newrefformat{chap}{Chapter~\ref{#1}}
\newrefformat{sec}{Section~\ref{#1}}
\newrefformat{algo}{Algorithm~\ref{#1}}
\newrefformat{fig}{Fig.~\ref{#1}}
\newrefformat{tab}{Table~\ref{#1}}
\newrefformat{rmk}{Remark~\ref{#1}}
\newrefformat{clm}{Claim~\ref{#1}}
\newrefformat{def}{Definition~\ref{#1}}
\newrefformat{cor}{Corollary~\ref{#1}}
\newrefformat{lmm}{Lemma~\ref{#1}}
\newrefformat{prop}{Proposition~\ref{#1}}
\newrefformat{app}{Appendix~\ref{#1}}
\newrefformat{hyp}{Hypothesis~\ref{#1}}
\newrefformat{thm}{Theorem~\ref{#1}}
\newrefformat{exam}{Example~\ref{#1}}

\newcommand{\pth}[1]{\left( #1 \right)}
\newcommand{\qth}[1]{\left[ #1 \right]}
\newcommand{\sth}[1]{\left\{ #1 \right\}}

\newcommand{\iprod}[2]{\left \langle #1, #2 \right\rangle}

\newcommand{\diag}[1]{\mathsf{diag} \left\{ {#1} \right\} }

\newcommand{\dTV}{\mathrm{TV}}

\newcommand{\tG}{{\widetilde{G}}}

\newcommand{\sfp}{{\mathsf{p}}}
\newcommand{\sfq}{{\mathsf{q}}}

\newcommand{\sfN}{{\mathsf{N}}}

\newcommand{\sfP}{{\mathsf{P}}}
\newcommand{\sfQ}{{\mathsf{Q}}}

\newcommand{\calE}{{\mathcal{E}}}
\newcommand{\calF}{{\mathcal{F}}}
\newcommand{\calG}{{\mathcal{G}}}
\newcommand{\calH}{{\mathcal{H}}}
\newcommand{\calI}{{\mathcal{I}}}
\newcommand{\calJ}{{\mathcal{J}}}

\newcommand{\calM}{{\mathcal{M}}}
\newcommand{\calN}{{\mathcal{N}}}

\newcommand{\calX}{{\mathcal{X}}}
\newcommand{\calY}{{\mathcal{Y}}}

\newcommand{\abs}[1]{\left| #1 \right|}

\begin{document}

\title{
The broken sample
problem revisited: Proof of a conjecture by Bai-Hsing and
high-dimensional extensions
}

\date{\today}

 \author{Simiao Jiao \and Yihong Wu \and Jiaming Xu\thanks{S. Jiao and J. Xu are with
 the Fuqua School of Business, Duke University, \texttt{\{simiao.jiao,jx77\}@duke.edu}. Y. Wu is with the Department of Statistics and Data Science, Yale University, \texttt{yihong.wu@yale.edu.}}}

\maketitle
\begin{abstract}
We revisit the classical broken sample problem: Two samples of i.i.d.\ data
points ${\mathbf{X}}=\{X_{1},\ldots , X_{n}\}$ and
${\mathbf{Y}}=\{Y_{1},\ldots ,Y_{m}\}$ are observed without correspondence
with $m\leq n$. Under the null hypothesis, ${\mathbf{X}}$ and
${\mathbf{Y}}$ are independent. Under the alternative hypothesis,
${\mathbf{Y}}$ is correlated with a random subsample of
${\mathbf{X}}$, in the sense that $(X_{\pi (i)},Y_{i})$'s are drawn independently
from some bivariate distribution for some latent injection
$\pi :[m] \to [n]$. Originally introduced by DeGroot, Feder, and Goel to
model matching records in census data, this problem has recently gained
renewed interest due to its applications in data de-anonymization, data
integration, and target tracking. Despite extensive research over the past
decades, determining the precise detection threshold has remained an open
problem even for equal sample sizes ($m=n$). Assuming $m$ and $n$ grow
proportionally, we show that the sharp threshold is given by a spectral
and an $L_{2}$ condition of the likelihood ratio operator, resolving a
conjecture of Bai and Hsing in the positive. These results are extended
to high dimensions and settle the sharp detection thresholds for Gaussian
and Bernoulli models.
\end{abstract}

\tableofcontents

\section{Introduction}
\label{sec:intro}

The broken sample problem~\cite{degroot1971matchmaking,degroot1980estimation,bai2005broken}
refers to the task of identifying the correlated records in two databases
that correspond to the same entity. Depending on the contexts, this problem
is also known as record linkage~\cite{dunn1946record}, data matching, feature
matching, database alignment. With origins in statistics and computer science
dating back to the 1960s~\cite{fellegi1969theory}, it was initially applied
to surveys, censuses, and data management tasks. In recent years, the problem
has gained significant attention~\cite{cullina2018fundamental,dai2019database,dai2020achievability,wang2022random}
due to emerging applications in data de-anonymization~\cite{narayanan2008robust}.
For example, anonymized datasets (e.g., Netflix) are matched to public
datasets (e.g., IMDb) to reveal the identities of anonymized users. Another
modern application arises in e-commerce, where identity fragmentation occurs
as consumers use multiple devices with different identifiers. Companies,
facing fragmented views of user behavior, can use effective linking of
browsing logs across devices to gain a comprehensive understanding of consumer
behavior, which is critical for marketing and advertising success~\cite{lin2022frontiers}.
The broken sample problem also has applications in particle tracking within
natural sciences, where the goal is to track mobile objects (e.g., birds
in flocks, motile cells, or particles in fluid) using consecutive video
frames~\cite{chertkov2010inference,semerjian2020recovery,moharrami2021planted,kunisky2022strong,ding2023planted}.
Furthermore, the problem is closely related to shuffled regression~\cite{pananjady2017linear,abid2017linear,unnikrishnan2018unlabeled,slawski2019linear,slawski2020two,mazumder2023linear,lufkin2024sharp},
where the pairings between covariates and response variables are missing
and must be reconstructed.

In this paper, instead of matching two correlated databases, we focus on
the more basic problem of deciding whether two databases
$\{X_{i}\}_{i=1}^{n}$ and $\{Y_{i}\}_{i=1}^{m}$ are correlated or not,
where $n \ge m$ denote the number of data points in two databases. More
formally, we formulate this question as a binary hypothesis testing problem.

\begin{definition}
[Broken sample detection]
\label{defn1}
Let ${\mathsf{P}}_{X,Y}$ denote a joint distribution over the space
${\mathcal{X}}\times {\mathcal{Y}}$ with its marginals
${\mathsf{P}}_{X}$ and ${\mathsf{P}}_{Y}$. Suppose we observe two datasets
$\{X_{i}\}_{i=1}^{n}$ and $\{Y_{i}\}_{i=1}^{m}$ distributed according to
one of the following hypotheses:

\begin{equation}
\label{eq:problem formulation}
\begin{aligned}
{\mathcal{H}}_{0}: & \quad \{X_{i}\}_{i=1}^{n} \, , \{Y_{i}\}_{i=1}^{m}
\sim {\mathsf{P}}_{X}^{\otimes n} \otimes {\mathsf{P}}_{Y}^{\otimes m};
\\
{\mathcal{H}}_{1}: & \quad \{(X_{\pi (i)},Y_{i})\}_{i=1}^{m} \, , \{X_{j}
\}_{j\in [n] \setminus \pi ([m])} \sim {\mathsf{P}}_{X,Y}^{\otimes m}
\otimes {\mathsf{P}}_{X}^{\otimes (n-m)},
\text{ conditional on $\pi \sim \text{Unif}(S_{m,n})$ }
\end{aligned}
\end{equation}
where $\otimes $ denotes the product measure and $S_{m,n}$ denotes the
set of injections from $[m]$ to $[n]$, $m\leq n$. The goal of
\emph{strong detection} is to distinguish between ${\mathcal{H}}_{0}$ and
${\mathcal{H}}_{1}$ with vanishing Type-I and Type-II error probabilities
as the dataset size $n \to \infty $.
\end{definition}

Note that if the underlying correspondence $\pi $ between $X_{i}$'s and
$Y_{i}$'s were observed, one could pair them accordingly and reduce the
problem to simply testing ${\mathsf{P}}_{X}\otimes {\mathsf{P}}_{Y}$ versus
${\mathsf{P}}_{X,Y}$ with $m$ iid observations. However, since the underlying
correspondence $\pi $ is unobserved, the inherent dependency between
$X_{i}$'s and $Y_{i}$'s are obscured by the missing pairing information,
making the testing problem significantly more challenging. In fact, an
effective test needs to be invariant to the permutation of $X_{i}$'s and
$Y_{i}$'s, and hence must depend only on the empirical distributions of
$X_{i}$'s and $Y_{i}$'s.

Note that the existing literature has been focusing on equal sample sizes
$m=n$, in which case the unknown correspondence $\pi $ is a permutation.
Our formulation above allows for $m<n$, to capture the practical scenarios
in which the datasets to be aligned have different sizes, such as the deanonymization
application in \cite{narayanan2008robust}. Clearly, unequal sample sizes
are more challenging, since a random portion of the $X$ sample (of unknown
location) is independent of the $Y$ sample, which further obscures the
signal.

Throughout the paper, we assume that the joint distribution
${\mathsf{P}}_{X,Y}$ is known. The case of unknown
${\mathsf{P}}_{X,Y}$ is much more challenging. In a companion paper~\cite{GongWuXu25},
we consider the problem of shuffled linear regression, where the $Y$ sample
(responses) is related to linearly the $X$ sample (covariates) through
the latent regression coefficients and the latent permutation of data points.

Here are two examples considered frequently in literature:

\begin{example}[Gaussian model]
\label{exam:Gaussian}
Let ${\mathcal{X}}= {\mathcal{Y}}= {\mathbb{R}}^{d}$ where
${\mathsf{P}}_{X,Y} =({\mathcal{N}}(\mu ,\Sigma ))^{\otimes d}$,
$\mu = (0,0)^ {\top}$, and
$\Sigma = (
\begin{smallmatrix}
1 &\rho
\\
\rho &1
\end{smallmatrix}
)$ with $\rho \in (0,1)$. Equivalently, under ${\mathcal{H}}_{0}$,
$\{X_{i}\}_{i=1}^{n}$ and $\{Y_{i}\}_{i=1}^{m}$ are simply two sets of
i.i.d.\ standard Gaussian random vectors in ${\mathbb{R}}^{d}$; under
${\mathcal{H}}_{1}$, $(X_{\pi (i)},Y_{i})_{i=1}^{m}$ are i.i.d.\ pairs
of standard Gaussian vectors with correlation $\rho $. The Gaussian case
with equal dataset size $m=n$ is considered
\cite{degroot1971matchmaking,bai2005broken} for low dimensions (fixed
$d$) and more recently by~\cite{elimelech2023phase,elimelech2024detection,paslev2023testing,tamir2025testing,zeynep2022detecting}
also for high dimensions ($d$ growing with $n$).
\end{example}

\begin{example}[Bernoulli model]
\label{exam:Bernoulli}
Let ${\mathcal{X}}= {\mathcal{Y}}= \{0,1\}^{d}$ and
${\mathsf{P}}_{X,Y}= {\mathsf{p}}^{\otimes d}$, where ${\mathsf{p}}$ denotes
the joint distribution of two Bernoulli random variables with success probability
$q \in [0,1]$ and correlation coefficient
$\rho \in [-\min \{\frac{q}{1-q},\frac{1-q}{q}\}, 1]$. Equivalently,
$X$ (resp.\ $Y$) can be identified with the adjacency matrix of a random
bipartite graph with $n$ (resp.\ $m$) left vertices and $d$ right vertices,
with edges formed independently with probability $q$. Under
${\mathcal{H}}_{0}$, the two graphs are independent. Under
${\mathcal{H}}_{1}$, the pair of edges $X_{\pi (i),j}$ and $Y_{i,j}$ are
$\rho $-correlated. This model has been studied by~\cite{paslev2023testing}
in the balanced case of $n=m$ and the dense regime for fixed $q$. In this
paper, we also consider unequal sizes $m<n$ and the sparse regime
$q=o(1)$.
\end{example}

The seminal work of Bai and Hsing~\cite{bai2005broken} studied the special
case where $m=n$, ${\mathcal{X}}={\mathcal{Y}}={\mathbb{R}}$, and
${\mathsf{P}}_{X,Y}$ is a fixed distribution independent of $n$. Letting
$L={\mathrm{d}}{\mathsf{P}}_{XY}/{\mathrm{d}}{\mathsf{P}}_{X}\otimes {\mathrm{d}}{
\mathsf{P}}_{Y}$ denote the likelihood ratio function, in~\cite[Theorem 1]{bai2005broken},
it is shown that if the likelihood ratio operator (see~Section~\ref{sec:main results}
for formal definition)
\begin{equation*}
({\mathcal{L}}h)(x) \triangleq \mathbb{E}_{Y\sim {\mathsf{P}}_{Y}}
\left [ L(x,Y)h(Y) \right ]
\end{equation*}
is square integrable (i.e.,
$\chi ^{2}({\mathsf{P}}_{X,Y}\|{\mathsf{P}}_{X}\otimes {\mathsf{P}}_{Y})
<\infty $) and its second-largest singular value $\lambda _{1} < 1$, strong
detection is impossible. The authors further conjectured that the converse
is also true, that is, strong detection is possible if
$\chi ^{2}({\mathsf{P}}_{X,Y}\|{\mathsf{P}}_{X}\otimes {\mathsf{P}}_{Y})
=\infty $ or $\lambda _{1} = 1$.

More recently, significant research has been devoted to the Gaussian model~\cite{zeynep2022detecting,elimelech2023phase,paslev2023testing,elimelech2024detection}
under the special case of equal sample sizes $m=n$. The state-of-the-art
results are summarized in~Table~\ref{tab:summarize_sota_gaussian}. Despite
these advancements, notable gaps still remain between the existing necessary
and sufficient conditions for strong detection. In particular, even the
correct scaling for the strong detection threshold remains unclear in the
simple setting with fixed dimension $d$.

\begin{table}[!ht]
\centering
\renewcommand{\arraystretch}{1.5}
\setlength{\tabcolsep}{10pt}
\begin{tabular}{|c|c|c|c|}
\hline
{} & {} & \textbf{State of the art $(m=n)$} & \textbf{This Work $(m=\Theta(n))$} \\ \cline{3-4} 
\hline
\multirow{2}{*}{\textbf{Fixed $d$}} 
& \textbf{Possible} & $\rho^2 = 1-o\left(n^{-(\frac{2}{d-1} \wedge 4)}\right)$ & $\rho^2 = 1-o(1)$ \\ \cline{2-4}
& \textbf{Impossible} & $\rho ^{2} = \rho ^{*}(d) - \Omega (1)$ & $\rho^2 = 1-\Omega(1)$ \\ \hline
\multirow{2}{*}{\textbf{Growing $d$}} 
& \textbf{Possible} & $\rho^2 = \omega\left(d^{-1}\right)$ & $\rho^2 = \omega\left(d^{-1}\right)$ \\ \cline{2-4}
& \textbf{Impossible} & $\rho^2 < (1-\epsilon)d^{-1}$ & $\rho^2 = O\left(d^{-1}\right)$ \\ \hline
\end{tabular}
\caption{Comparison of possibility and impossibility conditions for broken Gaussian samples. For $d\ge 4$, $\rho^*(d)$ is the unique root of $(1-x)^d=x$; for $d<4$, $\rho^*(d)=1/d$.}
\label{tab:summarize_sota_gaussian}
\end{table}

In this paper, we study a more general setting where ${\mathcal{X}}$ and
${\mathcal{Y}}$ can be any two possibly different spaces, and
${\mathsf{P}}_{X,Y}$ is an arbitrary joint distribution with potentially
non-identical marginals ${\mathsf{P}}_{X}$ and ${\mathsf{P}}_{Y}$. We further
allow the sample sizes to be unequal, focusing on the proportional regime
where $m=\alpha n$ for some constant $\alpha \in (0,1]$.

When ${\mathsf{P}}_{X,Y}$ does not depend on $n$, we show that
$\chi ^{2}({\mathsf{P}}_{X,Y}\|{\mathsf{P}}_{X}\otimes {\mathsf{P}}_{Y})
=\infty $, or $\lambda _{1} = 1$ and $\alpha =1$, are both sufficient (Theorem~\ref{thm:fixed P Q})
and necessary (Corollary~\ref{cor:lower bound general}) for strong detection,
resolving Bai-Hsing's conjecture in the positive and extending it to unequal
sample sizes. Specifically, for sufficiency, we construct a test statistic
with runtime linear in $n$ based on histogram approximations of the empirical
distributions of $X_{i}$'s and $Y_{i}$'s. For the case where
${\mathsf{P}}_{X,Y}$ depends on $n$,
\begin{itemize}
\item We show that if
$\chi ^{2}({\mathsf{P}}_{X,Y}\|{\mathsf{P}}_{X}\otimes {\mathsf{P}}_{Y})
= O(1)$ for $\alpha <1$, and in addition
$\lambda _{1} = 1 - \Omega (1)$ for $\alpha =1$, then strong detection
is impossible (see~Corollary~\ref{cor:lower bound general}).
\item We provide two computationally efficient tests based on the spectrum
of the likelihood ratio operator. Let
$1=\lambda _{0} \ge \lambda _{1} \ge \lambda _{2} \geq \cdots $ be the
singular values of the likelihood ratio operator ${\mathcal{L}}$. The first
test, using only the top eigenfunctions corresponding to
$\lambda _{1}$, achieves strong detection under the conditions
$\lambda _{1} \to 1$, $\alpha =1$, and certain moment assumptions (see~Theorem~\ref{thm:upper bound general eigen 1}).
The second test, leveraging the top-$r$ eigenfunctions and singular values
(where $r$ is appropriately chosen), succeeds in strong detection provided
that $\sum _{k=1}^{r} \lambda ^{2}_{k} \to \infty $ and certain additional
moment conditions (see~Theorem~\ref{thm:upper bound general eigen r}).
\end{itemize}

Particularizing these general results to the Gaussian model yields sharp
detection thresholds for both fixed and growing dimensions, as shown in~Table~\ref{tab:summarize_sota_gaussian}.
Additionally, we also settle the sharp detection threshold for the Bernoulli
case, closing the gaps in~\cite{paslev2023testing}.

Complementing our analysis of thresholds for strong detection, we also
examine the testing power of various computationally efficient tests in
the fixed distribution case as the sample size $n \to \infty $. Specifically,
we design a test leveraging the top-$r$ eigenfunctions and singular values
of the likelihood ratio operator with a runtime of $O(nr)$. We demonstrate
that its asymptotic power converges to that of the optimal likelihood ratio
test as $r\to \infty $. Numerical experiments with Gaussian samples further
reveal that even for moderately large $r$, the asymptotic power of the
proposed test closely approximates that of the optimal likelihood ratio
test.

\subsection{Notation}
\label{sec1.1}

We use ${\mathbb{N}}$ and ${\mathbb{R}}$ to denote natural number
$\{1,2,3,\cdots \}$ and real number, respectively. Let
${\mathbb{R}}_{+}$ denote positive real numbers. For
$n\in {\mathbb{N}}$, let $[n]\triangleq \{1,2,\cdots ,n\}$. Let
$S_{n}$ denote the set of all the permutations over $[n]$. For a given
permutation $\pi \in S_{n}$, $\pi (i)$ denotes the value to which
$\pi $ maps $i\in [n]$. Let $N_{k},k\in [n]$ denote the number of
$k$-cycles in the cycle decomposition of $\pi $. Let $\|v\|$ denote the
2-norm of vector $v$. Let $\mathbf I_{d} $ denote the identity matrix with
dimension $d$. The pseudo-inverse of a matrix $A$ is denoted by
$A^{\dag}$. The notation
$X_{1},\cdots ,X_{n} \stackrel{i.i.d.}{\sim} {\mathsf{P}}_{X}$
denotes that the random vectors $(X_{1},\cdots ,X_{n})$ are independent
and identically distributed (i.i.d.) according to ${\mathsf{P}}_{X}$. Let
${\mathcal{N}}(\mu ,\Sigma )$ denote the multivariate normal distribution
with mean vector $\mu $ and covariance matrix $\Sigma $. Let
$\chi ^{2}(d)$ denote the chi-square distribution with degrees of freedom
$d$.

For probability measures $P$ and $Q$, the total variation distance is
$\mathrm{TV}(P,Q)=\frac{1}{2}\int |dP-dQ|$, the squared Hellinger distance
is $H^{2}(P,Q) = \int (\sqrt{dP}-\sqrt{dQ})^{2}$, and the
$\chi ^{2}$-divergence is
$\chi ^{2}(P\|Q) = \int (\frac{dP}{dQ}-1)^{2} dQ$ if $P\ll Q$ and
$\infty $ otherwise.

We use standard big $O$ notations, e.g., for functions
$f,g:{\mathbb{N}}\to {\mathbb{R}}_{+}$, we say that $f = O(g)$ if there
exists $C>0$ such that $f(n) \leq Cg(n)$ for all $n$ large enough. We say
that $f = o(g)$ if $\lim _{n\to \infty} f(n)/g(n) = 0$. Let
$f = \Omega (g)$ if $g = O(f)$ and $f = \omega (g)$ if $g=o(f)$.

\subsection{Organization}
\label{sec1.2}

The rest of this paper is organized as follows.
Section~\ref{sec: related work} reviews related work in the literature.
Section~\ref{sec:main results} presents our main results, including possibility
and impossibility conditions. Section~\ref{sec: proof of lower bounds} contains
the proof of the impossibility result (Theorem~\ref{thm:expression of second moment} and~\prettyref{cor:lower bound general}).
Section~\ref{sec:pf-upper} provides the proofs of the sufficiency results
(Theorem~\ref{thm:fixed P Q},~\ref{thm:upper bound general eigen 1}, and~\ref{thm:upper bound general eigen r}).
Finally,~Section~\ref{sec: discussion} concludes the paper with a discussion
of open problems. Auxiliary lemmas and the proof of~Corollary~\ref{cor:Gaussian} and~\ref{cor:Bernoulli}
are presented in the~\hyperref[app:appendix]{Appendix}.

\section{Related work}
\label{sec: related work}

In this section, we discuss the related work and relevant results in detail.
These results all assume equal sample sizes $m=n$.

\subsection{Gaussian model}
\label{sec2.1}

The state-of-the-art findings for the Gaussian case are summarized in~Table~\ref{tab:summarize_sota_gaussian}.
In high dimensions where $d\to \infty $ as $n\to \infty $, a simple test
statistic based on $\sum _{i,j} X_{i}^{\top }Y_{j}$ achieves strong detection
if $\rho ^{2} = \omega (d^{-1})$, as shown in~\cite[Theorem 1]{zeynep2022detecting}.
Conversely, the best known impossibility result, established in~\cite[Theorem 4]{elimelech2023phase},
asserts that strong detection is unattainable if
$\rho ^{2} < (1-\epsilon )d^{-1}$ for some fixed $\epsilon > 0$.

An alternative test based on counting the number of pairs $(i,j)$ whose
pairwise correlation $X_{i}^{\top }Y_{j}$ exceeds a certain threshold is
proposed in~\cite{tamir2025testing}. In low dimensions with fixed
$d$, it has been shown that this count test achieves strong detection when
$\rho ^{2} = 1-o(n^{-\frac{2}{d-1}})$ for $d \ge 2$~\cite[Theorem 2]{elimelech2023phase}.
For $d=1$, a different count test was later shown in~\cite[Theorem 5]{paslev2023testing}
to achieve strong detection when $\rho ^{2} = 1-o(n^{-4})$. Conversely,
the best-known impossibility results assert that strong detection is impossible,
provided that $\rho ^{2} = 1-\Omega (1)$ for $d=1$~\cite[Theorem 1]{elimelech2024detection}
and $\rho ^{2} =\rho ^{\ast }(d) -\Omega (1)$ for $d\ge 4$, where
$\rho ^{\ast}(d)$ is the unique root of $(1-x)^{d}=x$~\cite[Theorem 4]{elimelech2023phase}
and $\rho ^{2} =1/d - \Omega (1)$ for $d = 2$ or $3$~\cite[Theorem 5]{elimelech2024detection}.

Despite these advances, the sharp threshold for strong detection remains
open for both low and high dimensions, which we resolve in this paper as
a corollary of our general results.

We mention in passing that the setting of fixed $n$ and
$d\to \infty $ has also been considered in the literature. It has been
shown that strong detection in this regime is possible if and only if
$\rho ^{2}d = \omega (1)$~\cite{zeynep2022detecting,elimelech2023phase}.

\subsection{Bernoulli model}
\label{sec2.2}

For the Bernoulli model, it is shown that for growing dimension $d$, strong
detection is possible if $\rho ^{2} = \omega (d^{-1})$~\cite[Example 2]{paslev2023testing}
and is impossible if
$\frac{\rho ^{2}}{1-\rho ^{2}} < (1-\epsilon ) d^{-1}$ for some fixed
$\epsilon > 0$~\cite[Theorem 2]{paslev2023testing}. For fixed dimension
$d$, strong detection is possible if $\rho ^{2} = 1-o(n^{-4/d})$~\cite[Example 4]{paslev2023testing}
and is impossible if $d < \log (1/\rho ^{2})/\log (1+\rho ^{2})$~\cite[Example 2]{paslev2023testing}.
This work sharpens these results by determining the precise thresholds
for both fixed and vanishing $q$.

\subsection{General distribution}
\label{sec2.3}

More recently, \cite{paslev2023testing} considers a generalization of the
Gaussian model in which ${\mathcal{X}}= {\mathcal{Y}}$ is a product space
of dimension $d$, and ${\mathsf{P}}_{X,Y}$ is a product distribution across
the $d$-dimensions with identical marginals. Notably,
\cite[Theorem 2]{paslev2023testing} establishes that strong detection is
impossible if
$\lambda _{1}^{2} \sum _{k=0}^{\infty }\lambda _{k}^{2} <1$ for fixed
$d$ and $\sum _{k=0}^{\infty }\lambda _{k}^{2} <e$ for growing $d$. These
results are much looser than those obtained in the present paper (See~Remark~\ref{rmk:comparison}
for details). For positive results,~\cite{paslev2023testing} considers
the generalized likelihood ratio test, which thresholds the maximum log-likelihood
over all permutations; the sum test, which thresholds the sum of centered
log-likelihood ratios; and the count test, which aggregates the number
of sample pairs whose likelihood ratios exceed a specific threshold. The
detection thresholds for these tests are characterized in terms of KL divergences
and Chernoff exponents, which yield looser bounds compared to the sharp
results derived in this work based on the $\chi ^{2}$-information.

\subsection{Weak detection}
\label{sec2.4}

In addition to the aforementioned strong detection, another well-studied
objective is \emph{weak detection}, which amounts to distinguishing between
${\mathcal{H}}_{1}$ and ${\mathcal{H}}_{0}$ with the sum of Type-I and
Type-II error probabilities asymptotically bounded away from $1$. In other
words, weak detection requires the testing algorithm to perform asymptotically
strictly better than random guessing.

Sharp thresholds for weak detection have been established in the Gaussian
case. Specifically, weak detection is shown~\cite[Theorem 3]{elimelech2023phase}
to be impossible when $\rho ^{2} = o(d^{-1})$. In the positive direction,
for growing dimension $d$, testing based on
$\sum _{i,j} X_{i}^{\top }Y_{j}$ was shown in~\cite[Theorem 1]{zeynep2022detecting}
to achieve weak detection, provided that
$\rho ^{2} = \Omega (d^{-1})$. For fixed dimension $d$, testing based on
$\|\bar{X}-\bar{Y}\|$ was shown in~\cite[Theorem 9]{elimelech2024detection}
to achieve weak detection when $\rho = \Omega (1)$.

Weak detection in general cases beyond the Gaussian setting is also studied
in~\cite{paslev2023testing}. It is shown in~\cite[Theorem 1]{paslev2023testing}
that when $\sum _{k=1}^{\infty }\lambda _{k}^{2} = o(1)$ weak detection
is impossible. This result coincides with~\cite[Theorem 3]{elimelech2023phase}
when specialized to the Gaussian case. For positive direction, a test based
on the summation of centered likelihood ratio proposed in~\cite[Theorem 4]{paslev2023testing}
achieves weak detection if a condition involving symmetric KL-divergence
between ${\mathsf{P}}_{X,Y}$ and
${\mathsf{P}}_{X}\otimes {\mathsf{P}}_{Y}$ holds. However, it remains open
whether weak detection is achievable when
$\sum _{k=1}^{\infty }\lambda _{k}^{2} =\Omega (1)$.

In the case where ${\mathsf{P}}_{X,Y}$ is a fixed distribution, we can
show that the sharp threshold for weak detection is
$\lambda _{1} > 0$. The negative part is trivial since likelihood ratio
operator $L$ is $1$ given $\lambda _{1} = 0$. Hence
${\mathsf{P}}_{X,Y} = {\mathsf{P}}_{X} \otimes{\mathsf{P}}_{Y}$. For the
positive part, we can consider the spectral statistic $T_{\mathrm{top}}$ in~(\ref{eq:statistic when lambda_1 to 1}).
Under CLT, $T_{\mathrm{top}}$ converges to $2\chi ^{2}(1)$ under
${\mathcal{H}}_{0}$ and
$\mathbb{E}_{1}[T_{\mathrm{top}}]=2(1-\lambda _{1})$, thereby achieving weak
detection when $\lambda _{1}>0$.

\subsection{Recovery thresholds}
\label{sec2.5}

Closely related to the hypothesis testing problem is the recovery problem
under the planted model ${\mathcal{H}}_{1}$, where the goal is to estimate
the latent permutation $\pi $ based on the $X$ and $Y$ samples. The Gaussian
model has received the most attention, for which the maximum likelihood
estimator (MLE) reduces to the following linear assignment problem:

\begin{equation}
\hat{\pi}_{\mathrm{ML}} \in \mathop{\arg \min }_{\pi \in S_{n}}
\frac{1}{n} \sum _{i=1}^{n} \| X_{\pi (i)} - Y_{i} \|^{2}.
\label{eq:MLE}
\end{equation}
The optimal objective value is the squared 2-Wasserstein distance between
the empirical distributions of $X_{i}$'s and $Y_{i}$'s. This linear assignment
model has been studied in the context of the broken sample problem since
1970s~\cite{degroot1971matchmaking,goel1975re,degroot1976matching,degroot1980estimation,bai2005broken}
and recently reintroduced as a model for database matching~\cite{dai2019database,dai2020achievability}
and as a geometric model for planted matching~\cite{kunisky2022strong}.
It is shown in~\cite{dai2019database} that the MLE coincides with the true
latent permutation $\pi $ with high probability, provided that
$\frac{d}{4}\log \frac{1}{1-\rho ^{2}} - \log n \to \infty $. Conversely,
it is shown in~\cite{wang2022random} that
$\frac{d}{4}\log \frac{1}{1-\rho ^{2}} - \log n +\log d \to \infty $ is
necessary for perfect recovery. Moreover, it is shown in~\cite{dai2020achievability}
that the maximum likelihood estimator achieves the almost perfect recovery
of $\pi $ with a vanishing fraction of errors, provided that
$\frac{d}{2}\log \frac{1}{1-\rho ^{2}} \ge (1+\epsilon ) \log n$ in the
high-dimensional regime of $d=\omega (\log n)$ and
$1-\rho ^{2} =o(n^{-2/d})$ in the low-dimensional regime of
$d=o(\log n)$~\cite{kunisky2022strong}, with matching lower bounds shown
in \cite{wang2022random}. To summarize and compare with the detection thresholds
in Table~\ref{tab:summarize_sota_gaussian}, we see that for
$d=\omega (\log n)$, the sharp thresholds for exact and almost exact recovery
are $\rho ^{2}d \geq 4 \log n$ and $\rho ^{2} d \geq 2 \log n$ respectively,
which differ from the detection threshold of $\rho ^{2}d=\omega (1)$ by
a logarithmic factor. For $d=\Theta (1)$, however, the recovery thresholds
are $1- \rho ^{2} =o( n^{-4/d}) $ and $1-\rho ^{2} =o( n^{-2/d})$ respectively,
which is much higher than $1-\rho ^{2}=o(1)$ for detection.

\section{Main results}
\label{sec:main results}

The sharp threshold for strong detection involves two information measures
which we now introduce. For a pair of random variables $(X,Y)$ with joint
distribution ${\mathsf{P}}_{X,Y}$ and marginals
${\mathsf{P}}_{X}, {\mathsf{P}}_{Y}$, their
\emph{$\chi ^{2}$-information} is defined as (see e.g.~\cite[Sec.~7.8]{polyanskiy2025information})
\begin{equation*}
I_{\chi ^{2}}(X; Y) \triangleq \chi ^{2}({\mathsf{P}}_{X,Y}\| {
\mathsf{P}}_{X}\otimes {\mathsf{P}}_{Y}).
\end{equation*}
This is the analog of the standard mutual information when the KL divergence
is replaced by the $\chi ^{2}$-divergence. The second quantity is the Hirschfeld-Gebelein-R\'enyi
\emph{maximal correlation} (see e.g.~\cite[Sec.~33.2]{polyanskiy2025information}),
defined as
\begin{equation*}
\rho (X; Y) \triangleq \sup _{g,h} \left \{\mathbb{E}_{{\mathsf{P}}_{X,Y}}
[g(X)h(Y)]: \mathbb{E}\left [ g(X) \right ] = \mathbb{E}\left [ h(Y)
\right ]= 0,\; \mathbb{E}\left [ g^{2}(X) \right ] = \mathbb{E}\left [
h^{2}(Y) \right ]=1 \right \}.
\end{equation*}

These two information measures have spectral interpretations in terms of
the conditional mean operator. Assume that the following likelihood ratio
for a single pair of observations exists:
\begin{equation}
\label{eq:likelihood ratio}
L(x,y) =
\frac{{\mathrm{d}}{\mathsf{P}}_{XY}}{{\mathrm{d}}{\mathsf{P}}_{X}\otimes {\mathsf{P}}_{Y}}(x,y),
\quad x\in {\mathcal{X}}, y\in {\mathcal{Y}}.
\end{equation}
This kernel defines a linear operator from $L_{2}({\mathsf{P}}_{Y})$ to
$L_{2}({\mathsf{P}}_{X})$:
\begin{equation}
\label{eq:H-S operator}
({\mathcal{L}}h)(x) =\mathbb{E}_{Y\sim {\mathsf{P}}_{Y}}\left [ L(x,Y)h(Y)
\right ] = \int L(x,y)h(y) {\mathrm{d}}{\mathsf{P}}_{Y}(y), \quad h \in L_{2}({
\mathsf{P}}_{Y}).
\end{equation}
In other words, $({\mathcal{L}}h)(x) = \mathbb{E}[h(Y)|X=x]$ under the
joint law ${\mathsf{P}}_{XY}$.

When $L \in L_{2}({\mathsf{P}}_{X} \otimes {\mathsf{P}}_{Y})$,
${\mathcal{L}}$ is a Hilbert-Schmidt operator and admits a spectral decomposition
(see~\cite[Equation (2)]{bai2005broken}):
\begin{equation}
\label{eq:SVD of L}
L(x,y) = \sum _{k=0}^{\infty }\lambda _{k} \phi _{k}(x)\psi _{k}(y),
\end{equation}
where $\{\phi _{k}\}_{k\geq 0}$ and $\{\psi _{k}\}_{k\geq 0}$ are orthonormal
bases of $L_{2}({\mathsf{P}}_{X})$ and $L_{2}({\mathsf{P}}_{Y})$ respectively,
and $\lambda _{0} \ge \lambda _{1} \ge \cdots $ are the singular values.
It is easy to check that $\|{\mathcal{L}}\| \le 1$, $\lambda _{0}=1$, and
$\phi _{0}=\psi _{0}\equiv 1$. As a result of this decomposition, we have
\begin{equation*}
I_{\chi ^{2}}(X;Y) =\|L\|^{2}_{{\mathsf{P}}_{X} \otimes {\mathsf{P}}_{Y}}
-1 = \sum _{k=1}^{\infty }\lambda _{k}^{2}, \quad \text{ and } \quad
\rho (X;Y)=\lambda _{1}.
\end{equation*}

Denoting $\mathbb{P}_{0}$ and $\mathbb{P}_{1}$ as the joint distributions
of
${\mathbf{X}}=(X_{1},\ldots ,X_{n}),{\mathbf{Y}}=(Y_{1},\ldots ,Y_{m})$
under ${\mathcal{H}}_{0}$ and ${\mathcal{H}}_{1}$, respectively. The Neyman-Pearson
lemma shows that the minimum sum of Type-I and Type-II errors is
$1-\mathrm{TV}(\mathbb{P}_{0},\mathbb{P}_{1})$, attained by the likelihood
ratio test
${\mathbf{1}_{\left \{{{\mathbf{L}}({\mathbf{X}},{\mathbf{Y}}) \geq 1}
\right \}}}$, where
\begin{equation}
{\mathbf{L}}({\mathbf{X}},{\mathbf{Y}}) \triangleq
\frac{{\mathrm{d}}\mathbb{P}_{1}}{{\mathrm{d}}\mathbb{P}_{0}}({\mathbf{X}},{
\mathbf{Y}}).
\label{eq:LRT}
\end{equation}
It turns out that the optimal likelihood ratio test
(\ref{eq:LRT}) is related to the kernel $L$ as follows. Denote by
$ \mathbb{P}_{1|\pi}$ the joint law of ${\mathbf{X}},{\mathbf{Y}}$ under
${\mathcal{H}}_{1}$ conditioned on the latent injection
$\pi \in S_{m,n}$, so that
$\mathbb{P}_{1} = \mathbb{E}_{\pi }\mathbb{P}_{1|\pi}$. Then
\begin{equation}
\label{eq:whole data set LR}
{\mathbf{L}}({\mathbf{X}},{\mathbf{Y}}) =
\frac{{\mathrm{d}}\mathbb{P}_{1}}{{\mathrm{d}}\mathbb{P}_{0}}({\mathbf{X}},{
\mathbf{Y}}) = \frac{1}{\left | S_{m,n} \right |}\sum _{\pi \in S_{m,n}}
\frac{{\mathrm{d}}\mathbb{P}_{1|\pi}}{{\mathrm{d}}\mathbb{P}_{0}}({\mathbf{X}},{
\mathbf{Y}}) = \frac{1}{\left | S_{m,n} \right |}\sum _{\pi \in S_{m,n}}
\prod _{i=1}^{m} L(X_{\pi (i)},Y_{i}).
\end{equation}
Directly computing this likelihood ratio requires to calculate the
\textit{permanent} of all $m\times m$ submatrices of the matrix whose
$(i,j)$-th entry is $L(X_{i}, Y_{j})$. The exact computation of permanent
for general matrices is NP-complete. In our case, since the matrix has
non-negative entries, there exists a polynomial-time (polynomial in
$m$ and $1/\epsilon $) Markov Chain Monte Carlo algorithm that, with high
probability, computes an approximation within a multiplicative factor of
$1\pm \epsilon $ from the true permanent~\cite{jerrum2004polynomial} for
each submatrix. However, the runtime scales as $\tilde{O}(m^{10})$, making
the algorithm impractical for large-scale datasets. Moreover, the likelihood
ratio test is difficult to analyze.

To overcome these statistical and computational challenges, we instead
bound the second moment of the likelihood ratio,
$\mathbb{E}_{0}\left [ {\mathbf{L}}^{2}({\mathbf{X}},{\mathbf{Y}})
\right ]$ to establish our impossibility results. For the positive direction,
we develop alternative, computationally efficient tests and show that they
attain the optimal detection thresholds.

\subsection{Impossibility results}
\label{sec:_impossibility_results_VTEX1}

In the following, we first provide an expression for
$\mathbb{E}_{0} {\mathbf{L}}^{2}$ and then obtain non-asymptotic impossible
result for detection that applies to general distribution
${\mathsf{P}}_{X,Y}$, which may depend on the sample sizes $m$ and
$n$. Here we use the convention that
$\binom{i}{i+1} = 0,\forall i \geq 0$ and $\binom{-1}{0} = 1$.

\begin{theorem}
\label{thm:expression of second moment}
For likelihood ratio ${\mathbf{L}}({\mathbf{X}},{\mathbf{Y}})$ corresponding
to problem~(\ref{eq:problem formulation}),
\begin{equation*}
\mathbb{E}_{0} {\mathbf{L}}^{2} = \sum _{\ell =0}^{m} t_{\ell }a_{
\ell},
\end{equation*}
where $t_{\ell }= \binom{n-\ell -1}{m-\ell}/\binom{n}{m}$ satisfying
$\sum _{\ell =0}^{m} t_{\ell}=1$ and $a_{\ell}$ is the $\ell $th coefficient
in the power series of
$\prod _{k=0}^{\infty }1/(1-z \lambda ^{2}_{k})$.
\end{theorem}

\begin{corollary}
\label{cor:lower bound general}
We have
\begin{equation*}
\mathbb{E}_{0} {\mathbf{L}}^{2} \leq \prod _{k=0}^{\infty }
\frac{1}{1-\frac{m}{n}\lambda ^{2}_{k}}.
\end{equation*}
If further $m = (1-o(1))n$, $\lambda _{1} \leq 1-c$ and
$\sum _{k=1}^{\infty }\lambda ^{2}_{k}\leq C$ for some constants
$c\in (0,1), C > 0$, then
\begin{equation*}
\left | \mathbb{E}_{0} {\mathbf{L}}^{2} - \prod _{k=1}^{\infty }
\frac{1}{1-\lambda ^{2}_{k}} \right | \to 0.
\end{equation*}

It follows that $\mathbb{E}_{0} {\mathbf{L}}^{2} = O(1)$ and hence strong
detection is impossible, if the following conditions hold,
\begin{itemize}
\item $m /n < 1-c_{1}$: $\frac{m}{n}I_{\chi ^{2}}(X;Y)\leq C_{1}$;
\item $m= (1-o(1))n$: $\rho (X;Y) \leq 1-c$ and $I_{\chi ^{2}}(X;Y)
\leq C$,
\end{itemize}
for some constants $c,c_{1} \in (0,1), C,C_{1} >0$.
\end{corollary}

\begin{remark}
\label{rem1}
In the special case of $m=n$,
$\mathbb{E}_{0} {\mathbf{L}}^{2}=a_{n}$ which converges to
$ \prod _{k=1}^{\infty }1/(1-\lambda _{k}^{2})$. This result has been proved
in~\cite[Theorem 1]{bai2005broken} for fixed ${\mathsf{P}}_{X,Y}$ and for
${\mathcal{X}}= {\mathcal{Y}}= {\mathbb{R}}$ using a beautiful argument
based on P\'olya's Theorem~\cite[Theorem 7.7]{stanley2018algebraic} for
the cycle index polynomial of permutation group and the Cauchy integral
theorem. We extend this argument to allow unequal sample sizes
$m \le n$. The key observation involves a bipartite multigraph (cf.~Fig.~\ref{fig:2core})
with left vertices $[m]$ and right vertices $[n]$, with edges
$(i,i)$ and $(i,\pi (i))$ for $i \in [m]$, where $\pi $ is the latent random
injection. It turns out the second moment calculation only depends on the
action of $\pi $ restricted to the 2-core of this random graph, on which
$\pi $ acts as a permutation. Finally, we average over the size of this
2-core whose probability mass function is given by the sequence
$t_{\ell}$ in Theorem~\ref{thm:expression of second moment}.
\end{remark}

\begin{remark}
\label{rmk:comparison}
Similar but looser impossibility conditions have been derived in the prior
work~\cite{paslev2023testing} for the special case where $m=n$,
${\mathcal{X}}= {\mathcal{Y}}$ is a product space with dimension $d$, and
${\mathsf{P}}_{X,Y}$ is a product distribution across the $d$ dimensions
with identical marginals. Specifically, after translating into our notation,
the impossibility conditions in~\cite{paslev2023testing} are given by
$\rho ^{2}_{1}(X;Y) \left (1+ I_{\chi ^{2}}(X;Y)\right ) <1$ for fixed
$d$ and $I_{\chi ^{2}}(X;Y) < e-1$ for growing $d$. These conditions are
also derived by first expressing the second moment in terms of the cycle
index polynomial of the permutation group, similar to our approach. However,
unlike our analysis, which computes the exact asymptotic value of the cycle
index polynomial using P\'olya's Theorem and the Cauchy integral theorem,
\cite{paslev2023testing} obtains only loose upper bounds via Poisson approximation.
\end{remark}

\subsection{Test statistics and positive results}
\label{sec:test and positive results}

Before presenting sufficient conditions on strong detection, we discuss
the common rationale underlying the construction of various test statistics
in this paper. As mentioned in Section~\ref{sec:intro}, absent any correspondence
between the $X$ and $Y$ samples, it is statistically sufficient to summarize
them into empirical distributions
$\hat{P}_{X} \triangleq \frac{1}{n}\sum _{i=1}^{n} \delta _{X_{i}}$ and
$\hat{P}_{Y} \triangleq \frac{1}{m}\sum _{i=1}^{m}\delta _{Y_{i}}$. To
test the correlation between the data points, a natural idea is to pick
a pair of test functions that embed the sample space into a Euclidean space,
compute their empirical average on the $X$ and $Y$ samples respectively,
and test the \textit{independence} of these two random vectors.

Specifically, fix test functions
$f: {\mathsf{X}}\to {\mathbb{R}}^{k}$ and
$g: {\mathsf{Y}}\to {\mathbb{R}}^{k}$ such that
$\mathbb{E}_{{\mathsf{P}}_{X}}[f(X)] = \mathbb{E}_{{\mathsf{P}}_{Y}}[g(Y)]
= 0$. Let
\begin{equation*}
\bar{f}_{n} \equiv \frac{1}{\sqrt{n}} \sum _{i=1}^{n} f(X_{i}),
\quad \bar{g}_{m} \equiv \frac{1}{\sqrt{m}} \sum _{i=1}^{m} g(Y_{i}).
\end{equation*}
Under the classic large-sample asymptotics (namely, with fixed
${\mathsf{P}}_{X,Y}$ and $m,n\to \infty $), by the Central Limit Theorem
(CLT), the limiting joint law of $(\bar{f}_{n}, \bar{g}_{m})$ is independent
centered multivariate normal under $H_{0}$ and correlated normal under
$H_{1}$. We can then apply the optimal test for distinguishing two Gaussians
with different covariance matrices, known as the
\textit{quadratic discriminant analysis} (QDA). In settings where
${\mathsf{P}}_{X,Y}$ may depend on $n$ (e.g.~, through the ambient dimensions),
we can no longer rely on CLT; instead, we directly use the distance or
correlation between $\bar{f}_{n}$ and $\bar{g}_{m}$ as test statistic.

We will employ two classes of test functions: \textit{histogram} (empirical
frequencies on a finite partition of the sample space) and
\textit{spectral} (top few eigenfunctions in the spectral decomposition
(\ref{eq:SVD of L})).

\subsubsection{Fixed distribution}
\label{sec3.2.1}

For fixed ${\mathsf{P}}_{X,Y}$, the following theorem provides a sharp
characterization of strong detection.

\begin{theorem}
\label{thm:fixed P Q}
For problem~(\ref{eq:problem formulation}), assume that
${\mathsf{P}}_{X,Y}$ does not depend on $n$ and $m = \alpha n$ for some
$\alpha \in (0,1]$. Strong detection is possible if and only if
\begin{itemize}
\item $I_{\chi ^{2}}(X;Y) = \infty $ or $\rho (X;Y)=1$ for
$\alpha = 1$;
\item $I_{\chi ^{2}}(X;Y) = \infty $ for $\alpha < 1$.
\end{itemize}
\end{theorem}

The necessity part of~Theorem~\ref{thm:fixed P Q} directly follows from~Corollary~\ref{cor:lower bound general}.
The challenging part is the sufficiency, which is left in~\cite{bai2005broken}
as a conjecture. Since $I_{\chi ^{2}}(X;Y)$ may be infinity, the operator
${\mathcal{L}}$ as defined in~(\ref{eq:H-S operator}) may not be Hilbert-Schmidt
and may not admit a spectral decomposition.

To overcome this challenge, we leverage the following crucial observation.
When $X_{i}$'s and $Y_{i}$'s are real-valued data, the empirical cumulative
distribution functions have Gaussian-like fluctuations. These fluctuations
contain all the information there is for detection. More specifically,
let
$\hat{F}_{X}(t) = \frac{1}{n} \sum _{i=1}^{n} {\mathbf{1}_{\left \{{X_{i}
\le t}\right \}}} $ and
$\hat{F}_{Y}(t) = \frac{1}{m} \sum _{i=1}^{m} {\mathbf{1}_{\left \{{Y_{i}
\le t}\right \}}} $ denote the empirical CDFs. Let $F_{X}(t)$ and
$F_{Y}(t)$ denote the true CDFs. Then
\begin{equation*}
\left ( \sqrt{n}(\hat{F}_{X}(t) - F_{X}(t)), \sqrt{m}( \hat{F}_{Y}(t) -
F_{Y}(t)) \right ) \to \left ( B_{X}(t),B_{Y}(t)\right ),
\end{equation*}
where $(B_{X}, B_{Y})$ are two Brownian bridges that are independent under
${\mathcal{H}}_{0}$ and correlated under ${\mathcal{H}}_{1}$. (See
\cite[Sec.~2.2]{ding2021efficient} about this covariance structure.)

This observation motivates us to design tests based on the histograms,
following~\cite{ding2021efficient}. We first appropriately partition the
space ${\mathcal{X}}$ into disjoint regions $\{I_{k}\}_{k=1}^{w}$ and
${\mathcal{Y}}$ into $\{J_{l}\}_{l=1}^{w}$. By the variational representation
of divergences, we can ensure that the $\chi ^{2}$-divergence between discretized
${\mathsf{P}}_{X,Y}$ and ${\mathsf{P}}_{X}\otimes {\mathsf{P}}_{Y}$ is
sufficiently large. We then compute the empirical frequency of
$X_{i}$'s and $Y_{i}$'s in each region to obtain a histogram statistic.
Specifically, define the standardized histogram statistics
\begin{align}
s_{k} = \frac{1}{\sqrt{n}} \sum _{i=1}^{n}
\frac{{\mathbf{1}_{\left \{{X_{i} \in I_{k}}\right \}}} - r_{k}}{ \sqrt{r_{k}}},
\quad t_{l}= \frac{1}{\sqrt{m}} \sum _{i=1}^{m}
\frac{{\mathbf{1}_{\left \{{Y_{i} \in J_{l}}\right \}}} - z_{l}}{\sqrt{z_{l}}},
\label{eq:histogram vectors}
\end{align}
where $r_{k} = {\mathsf{P}}_{X}(I_{k})$ for $k\in [w]$ and
$z_{l} = {\mathsf{P}}_{Y}(J_{l})$ for $l\in [w]$.

Let $s=(s_{1},\ldots , s_{w})$ and $t=(t_{1},\ldots , t_{w})$. By the CLT,
as $n \to \infty $, $(s,t)$ converges in distribution to
${\mathcal{N}}(\mathbf 0, \Sigma _{0})$ and
${\mathcal{N}}(\mathbf 0, \Sigma _{1})$ under ${\mathcal{H}}_{0}$ and
${\mathcal{H}}_{1}$, respectively, where
\begin{equation}
\Sigma _{0}=
\begin{bmatrix}
A & \mathbf 0
\\
\mathbf 0& B
\end{bmatrix}
, \quad \Sigma _{1}=
\begin{bmatrix}
A & \sqrt{\alpha}C
\\
\sqrt{\alpha}C^{\top }& B
\end{bmatrix}
,
\label{eq:Cov}
\end{equation}
where $A = \mathbf I_{w} - \gamma \gamma ^{\top }$ with
$\gamma _{i} = \sqrt{r_{i}}$,
$B =\mathbf I_{w} - \beta \beta ^{\top }$ with
$\beta _{i} = \sqrt{z_{i}}$, and
$C ={\widetilde{P}}-\gamma \beta ^{\top }$ with
${\widetilde{P}}_{k,l} \equiv{\mathsf{P}}_{X,Y}( I_{k}\times J_{l})/
\sqrt{r_{k}z_{l}}$. Crucially, the left (resp.~right) singular vectors
of $C$ are given by those of $A$ (resp.~$B$). This reduces the problem
of distinguishing two Gaussians with different covariance matrices and
we can apply the optimal likelihood ratio test under these limiting Gaussian
laws. In the non-degenerate case of $\rho (X;Y)<1$,\footnote{This condition
ensures that $\Sigma _{1}$ has all but one positive eigenvalues. Indeed,
one can show that eigenvalues of $\Sigma _{1}$ equals
$1\pm \sqrt{\alpha} \mu _{k}$, where $\mu _{k}$'s are the singular values
of the discretized likelihood ratio operator, which by definition is at
most $\rho (X;Y)$. See Section~\ref{sec:pf-upper} for details.} it reduces
to a QDA test: reject ${\mathcal{H}}_{0}$ if
\begin{align}
\label{eq:histogram test}
T_{\mathrm{{hist}}} \triangleq -\frac{1}{2}
\begin{bmatrix}
s^{\top }& t^{\top }
\end{bmatrix} \left ( \Sigma _{1}^{\dag }- \Sigma _{0}^{\dag }
\right )
\begin{bmatrix}
s
\\
t
\end{bmatrix}- \frac{1}{2}\sum _{k=1}^{w} \log (1-\alpha \mu _{k}^{2}),
\end{align}
is positive. Here $\Sigma _{0}^{\dag}, \Sigma _{0}^{\dag}$ are pseudoinverses
because the covariance matrices are rank-deficient (histogram normalizes)
and $\{\mu _{k}\}_{k=1}^{w}$ are the singular values of the off-diagonal
part $C$ in (\ref{eq:Cov}).

Finally, we bound the TV distance between
${\mathcal{N}}(\mathbf 0, \Sigma _{0})$ and
${\mathcal{N}}(\mathbf 0, \Sigma _{1})$ in terms of the Hellinger distance,
which, thanks to normality, can be further bounded in terms of the
$\chi ^{2}$-divergence between the discretized ${\mathsf{P}}_{X,Y}$ and
${\mathsf{P}}_{X}\otimes {\mathsf{P}}_{Y}$. Since this $\chi ^{2}$-divergence
can be made arbitrarily large by choosing the appropriate partition, we
conclude that the test error of $T_{\mathrm{{hist}}}$ can be arbitrarily small,
achieving strong detection.

\subsubsection{General distribution}
\label{sec3.2.2}

Next, we consider the case where ${\mathsf{P}}_{X,Y}$ depends on $n$, but
$I_{\chi ^{2}}(X;Y) <\infty $, so that the likelihood ratio operator
${\mathcal{L}}$ admits the spectral decomposition as per~(\ref{eq:SVD of L}).
Thus we can design a test based on eigenfunctions and eigenvalues. As described
at the beginning of this section, we use the top $r$ eigenfunctions as
spectral embedding:
\begin{equation}
\Phi _{r} \equiv \Phi _{r}({\mathbf{X}}) \triangleq
\frac{1}{\sqrt{n}} \sum _{i=1}^{n}
\begin{bmatrix}
\phi _{1}(X_{i})
\\
\vdots
\\
\phi _{r}(X_{i})
\end{bmatrix}
, \quad \Psi _{r} \equiv \Psi _{r}({\mathbf{Y}}) \triangleq
\frac{1}{\sqrt{m}} \sum _{i=1}^{m}
\begin{bmatrix}
\psi _{1}(Y_{i})
\\
\vdots
\\
\psi _{r}(Y_{i})
\end{bmatrix}
\label{eq:eigenvector test}
\end{equation}
and then compute either their distance or inner product as test statistics.

Suppose the maximal correlation $\rho (X; Y)$ is high, in which case the
top eigenfunctions are close, we may simply compute their squared difference:
\begin{equation}
\label{eq:statistic when lambda_1 to 1}
T_{\mathrm{{top}}}\triangleq (\Phi _{1}-\Psi _{1})^{2}=\left (
\frac{1}{\sqrt{n}} \sum _{i=1}^{n} \phi _{1}(X_{i}) -
\frac{1}{\sqrt{m}} \sum _{i=1}^{m} \psi _{1}(Y_{i}) \right )^{2}.
\end{equation}
The following theorem shows that when $\rho (X; Y) \to 1$ and
$m/n \to 1$ (and some mild moment conditions hold), strong detection is
achieved by rejecting ${\mathcal{H}}_{0}$ when $ T_{\mathrm{{top}}}$ is smaller
than an appropriately chosen threshold.

\begin{theorem}
\label{thm:upper bound general eigen 1}
Assume that $I_{\chi ^{2}}(X;Y) < \infty $ for any $n$ and
$m=(1-o(1))n$. If either
$\mathbb{E}_{{\mathsf{P}}_{X}} |\phi ^{3}_{1}(X)| = o(\sqrt{n})$ or
$\mathbb{E}_{{\mathsf{P}}_{Y}} |\psi ^{3}_{1}(Y)| = o(\sqrt{m})$,
$T_{\mathrm{{top}}}$ in~(\ref{eq:statistic when lambda_1 to 1}) achieves
strong detection, provided that
\begin{equation}
\label{eq:_maximal_correlation_goes_to_1_unequal_data_VTEX1}
\lim _{n\to \infty} \rho (X;Y) = 1.
\end{equation}
\end{theorem}

When $\rho (X;Y)=1-\Omega (1)$ or $m/n=1-\Omega (1)$, we need to incorporate
higher-order eigenpairs. To this end, we compute the weighted inner product
between the spectral embeddings with weights given by the eigenvalues:
\begin{equation}
T_{\mathrm{{inner}}}\triangleq \Phi _{r}^{\top }\mathsf{diag} \left \{ {
\lambda _{1},\cdots ,\lambda _{r}} \right \} \Psi _{r}
\label{eq:Tinner}
\end{equation}
Alternatively, define
\begin{equation*}
L_{r}(x,y) = \sum _{k=1}^{r} \lambda _{k} \phi _{k}(x) \psi _{k}(y),
\end{equation*}
which is a truncated version of the likelihood ratio $L$ in
(\ref{eq:SVD of L}). (Note that $L_{\infty}=L - 1$.) We also define
the rank-$r$ truncated version of the $\chi ^{2}$-information as the second
moment of $L_{r}$ under ${\mathsf{P}}_{X} \otimes {\mathsf{P}}_{Y}$.
\begin{equation*}
I_{\chi ^{2}}^{(r)}(X;Y) \triangleq \mathbb{E}_{(X,Y)\sim {\mathsf{P}}_{X}
\otimes {\mathsf{P}}_{Y}}\left [ L^{2}_{r}(X,Y) \right ]= \sum _{k=1}^{r}
\lambda ^{2}_{k}.
\end{equation*}
Then (\ref{eq:Tinner}) can be rewritten as:
\begin{equation}
\label{eq:m/n sum lambda^2_k goes to infinity}
T_{\mathrm{{inner}}}= \frac{1}{\sqrt{nm}} \sum _{i=1}^{n} \sum _{j=1}^{m}L_{r}(X_{i},Y_{j})
= \frac{1}{\sqrt{nm}} \sum _{i=1}^{n}\sum _{j=1}^{m} \sum _{k=1}^{r}
\lambda _{k} \phi _{k}(X_{i})\psi _{k}(Y_{j}).
\end{equation}

The next result shows that $T_{\mathrm{{inner}}}$ achieves strong detection provided
that $I_{\chi ^{2}}^{(r)}(X;Y)\to \infty $ and some additional moment condition
holds.

\begin{theorem}
\label{thm:upper bound general eigen r}
Assume that $I_{\chi ^{2}}(X;Y) < \infty $ for any $n$. Suppose that
for some $r$ (possibly depending on $n$)
$\mathbb{E}_{{\mathsf{P}}_{X,Y}} L^{2}_{r}(X,Y) = o\left ( m I_{\chi ^{2}}^{(r)}(X;Y)^{2}
\right )$. Then $T_{\mathrm{{inner}}}$ in~(\ref{eq:m/n sum lambda^2_k goes to infinity})
achieves strong detection, provided that:
\begin{equation*}
\lim _{n\to \infty} \frac{m}{n}I_{\chi ^{2}}^{(r)}(X;Y) = \infty
\end{equation*}
\end{theorem}

\begin{remark}
\label{rem3}
Suppose $\rho (X;Y) \le 1-\Omega (1)$. If the moment condition were to
hold for $r=\infty $, then the sufficient condition
$(m/n) I_{\chi ^{2}}^{\infty}(X;Y) \to \infty $ would be tight, since
$I_{\chi ^{2}}^{\infty}(X;Y)=I_{\chi ^{2}}(X;Y)$ and strong detection is
impossible if $(m/n) I_{\chi ^{2}}(X;Y) =O(1)$ as established in~Corollary~\ref{cor:lower bound general}.
However, the moment condition may not hold for large $r$, preventing us
from obtaining a tight sufficient condition in general cases.
\end{remark}

\subsection{Gaussian and Bernoulli models}
\label{sec:_gaussian_and_bernoulli_models_VTEX1}

As applications of our general results (notably,
Corollary~\ref{cor:lower bound general},~Theorem~\ref{thm:upper bound general eigen 1}
and~\ref{thm:upper bound general eigen r}), we determine the sharp detection
threshold for the Gaussian and Bernoulli model, in both low-dimensional
and high-dimensional regimes. We focus on the proportional regime of
$m = \alpha n$ for some constant $\alpha \in (0,1]$.

\begin{corollary}
\label{cor:Gaussian}
For the Gaussian model in~Example~\ref{exam:Gaussian} with
$m=\alpha n$, the necessary and sufficient condition for strong detection
is as follows:
\begin{itemize}
\item Fixed $d$: $\rho ^{2} = 1 -o(1)$.
\item Growing $d$: $\rho ^{2} d = \omega (1)$.
\end{itemize}
\end{corollary}

We discuss the test statistics used in Corollary~\ref{cor:Gaussian}. For growing
$d$, we consider $T_{\mathrm{{inner}}}$ by applying
Theorem~\ref{thm:upper bound general eigen r} with $r = d$. For fixed
$d$, one can still apply Theorem~\ref{thm:upper bound general eigen r} with
$r$ slowly growing with $n$ to establish the upper bound. However, verifying
the second moment condition is technically involved. Instead, we directly
apply $T_{\mathrm{eigen}}$ defined in (\ref{eq:Teigen}). Since this statistic
corresponds to the QDA test between two Gaussians, we can leverage the
Gaussian limits of the vector $(\Phi _{r},\Psi _{r})$ to provide a more
concise proof.

\begin{remark}
\label{rmk:unbalanced}
Corollary~\ref{cor:Gaussian} establishes the sharp threshold for strong detection
when $m = \Theta (n)$. However, for highly unbalanced sample sizes
$m = o(n)$, the sharp threshold remains open even in low dimensions. Specifically,
recall from~Corollary~\ref{cor:lower bound general} that strong detection is
impossible if $(1-\rho ^{2})^{-d} m/n \leq C$ for some absolute constant
$C>0$. It is unclear whether strong detection is achievable when
$(1-\rho ^{2})^{-d} m/n \to \infty $. Applying~Theorem~\ref{thm:upper bound general eigen r}
with $r=d$ in the Gaussian case only yields the sufficient condition
$\rho ^{2} d m/n \to \infty $.
\end{remark}

It is interesting to note that for equal sample sizes $m=n$ (so
$\alpha = 1$), a simple test based on the sample means attains the optimal
threshold in both low and high dimensions. Let
\begin{equation*}
\bar X \triangleq \frac{\sum _{i=1}^{n} X_{i}}{\sqrt{n}},\quad \bar Y
\triangleq \frac{\sum _{i=1}^{m} Y_{i}}{\sqrt{m}},
\end{equation*}
which are jointly distributed as
${\mathcal{N}}(\mathbf 0,\mathbf I_{2d})$ under ${\mathcal{H}}_{0}$ and
${\mathcal{N}}\left (\mathbf 0,[
\begin{smallmatrix}
\mathbf I_{d} & \sqrt{\alpha}\rho \mathbf I_{d}
\\
\sqrt{\alpha}\rho \mathbf I_{d} & \mathbf I_{d}
\end{smallmatrix}
]\right )$ under ${\mathcal{H}}_{1}$. The optimal likelihood ratio test
between these two Gaussians is the following QDA test: reject
${\mathcal{H}}_{0}$ if
\begin{align}
T_{\text{\textup{means}}}= - \frac{\alpha \rho ^{2}}{2(1-\alpha \rho ^{2})}
\|\bar X-\bar Y\|^{2} + \frac{\sqrt{\alpha}\rho}{1+\sqrt{\alpha}\rho}
\left \langle \bar X, \bar Y \right \rangle - \frac{d}{2} \log (1-
\alpha ^{2}\rho ^{2})
\label{eq:LRT_sample_means}
\end{align}
is positive.

One can verify that when $\alpha = 1$, the total variation distance between
those two Gaussians converges to $1$ when $\rho ^{2}d \to \infty $ for
growing $d$ or $\rho ^{2} \to 1$ for fixed $d$, which are precisely the
optimal detection threshold.\footnote{In fact, instead of
(\ref{eq:LRT_sample_means}), it is sufficient to consider the simpler
test statistic $\|\bar X-\bar Y\|^{2}$ for low dimensions ($\rho $ close
to $1$) and $\langle \bar X, \bar Y \rangle $ for high dimensions ($
\rho $ close to 0). The latter is the same test considered in
\cite[Theorem 1]{zeynep2022detecting}.} Somewhat surprisingly, for
$\alpha < 1$, test based on sample means only succeeds in high dimensions.
Indeed, if $d$ is fixed, the joint distributions of
$(\bar X,\bar Y)$ are not perfectly distinguishable even if
$\rho =1$ because averaging over two unequal samples destroyed their perfect
correlation. It turns out that sample means correspond to linear spectral
embedding and incorporating higher-order spectral embedding (Hermite polynomials)
attains the optimal detection threshold. We postpone this discussion to
the next section -- see (\ref{eq:Teigen}).

Next, we turn to the Bernoulli model in~Example~\ref{exam:Bernoulli}.
\begin{corollary}
\label{cor:Bernoulli}
For the Bernoulli model with $m=\alpha n$ and $q \le 1/2$,\footnote{Since
we can interchange $0$ and $1$, the Bernoulli model with parameter
$q$ is equivalent to the one with $1-q$. Therefore, we assume
$q \le 1/2$ without loss of generality.} the sufficient and necessary condition
for strong detection is as follows:

\begin{itemize}
\item Fixed $d$: $\alpha = 1$, $\rho ^{2} = 1 -o(1)$, and
$nq = \omega (1)$.
\item Growing $d$: $\rho ^{2} d = \omega (1)$ and
$ndq\left | \rho \right | = \omega (1)$.
\end{itemize}
\end{corollary}

\begin{remark}
\label{rem5}
For growing $d$, the optimal threshold is achieved by the inner-product
test~(\ref{eq:m/n sum lambda^2_k goes to infinity}) with $r=d$, which
simplifies to $\left \langle \bar{X}, \bar{Y} \right \rangle $ and is equivalent
to the test considered in \cite[Example 4]{paslev2023testing}. For fixed
$d$, previous literature did not identify the sharp threshold, which we
obtain by using the test~(\ref{eq:statistic when lambda_1 to 1}). Additionally,
Corollary~\ref{cor:Bernoulli} allows for $\alpha <1$ and $q=o(1)$, which has
not been considered in prior work. It is worth noting that the necessary
conditions $nq=\omega (1)$ for fixed $d$ and $ndq|\rho |=\omega (1)$ for
growing $d$ are derived by considering an easier instance in which the
latent injection $\pi $ is known. In this case, the problem simplifies
to testing ${\mathsf{P}}_{X,Y}$ vs
${\mathsf{P}}_{X}\otimes {\mathsf{P}}_{Y}$ based on $m$ independent observations,
for which
$m H^{2}({\mathsf{P}}_{X,Y}, {\mathsf{P}}_{X} \otimes {\mathsf{P}}_{Y})
\to \infty $ is known to be necessary for strong detection.
\end{remark}

\subsection{Power analysis and experiments}
\label{sec:power analysis}

In the preceding sections, we focus on establishing sufficient and necessary
conditions for strong detection with vanishing Type-I and Type-II errors.
In this section, we analyze the power of various tests in the special case
of fixed distribution, where $I_{\chi ^{2}}(X;Y) < \infty $ and
$\rho (X;Y) < 1$ so that the Type-I $+$ II errors are bounded away from
$0$. To this end, it suffices to describe the limiting distribution of
test statistics.

\subsubsection{Limiting distribution}
\label{sec3.4.1}

Recall that the optimal test is the likelihood ratio test
(\ref{eq:whole data set LR}), which is challenging to compute. A beautiful
result by Bai and Hsing~\cite[Theorem 2]{bai2005broken} (which originally
focused on $\alpha = 1$ but can be extended to $\alpha \in (0,1)$) shows
that under the null hypothesis ${\mathcal{H}}_{0}$, the log-likelihood
ratio $\log {\mathbf{L}}({\mathbf{X}},{\mathbf{Y}})$ converges to the following
random variable $\xi $ in distribution as $n\to \infty $ for fixed
${\mathsf{P}}_{X,Y}$, where
\begin{align}
\xi = - \frac{1}{2}\sum _{k=1}^{\infty }\left (
\frac{\sqrt{\alpha}\lambda _{k}}{1-\sqrt{\alpha}\lambda _{k}} U_{k}^{2}
-\frac{\sqrt{\alpha}\lambda _{k}}{1+\sqrt{\alpha}\lambda _{k}} V_{k}^{2}
+ \log \left ( 1-\alpha \lambda _{k}^{2}\right ) \right )
\label{eq:def_xi}
\end{align}
with the $U_{k}, V_{k}$ denoting iid.\ standard Gaussian random variables.
Notably, this limiting distribution only depends on
${\mathsf{P}}_{X,Y}$ through the singular values $\{\lambda _{k}\}$ of
the likelihood ratio kernel (\ref{eq:likelihood ratio}).

Next, we proceed to computationally feasible tests based on either spectral
or histogram embeddings. As done in ~(\ref{eq:LRT_sample_means}), we
construct a test by summarizing the two samples into the spectral embedding
$\Phi _{r}$ and $\Psi _{r}$ in (\ref{eq:eigenvector test}). Using their
asymptotic normality, instead of (\ref{eq:Tinner}), we apply the QDA
test in order to maximize the power. By the CLT, for any fixed $r$,
$\left (\Phi _{r}, \Psi _{r}\right )$ converges in distribution to
${\mathcal{N}}(\mathbf 0,\mathbf I_{2r})$ under ${\mathcal{H}}_{0}$ and
${\mathcal{N}}(\mathbf 0, \Sigma _{2r} )$ under ${\mathcal{H}}_{1}$, as
$n \to \infty $, where
$\Sigma _{2r}= \left [
\begin{smallmatrix}
\mathbf I_{r} & \sqrt{\alpha}\Lambda _{r}
\\
\sqrt{\alpha}\Lambda _{r} & \mathbf I_{r}
\end{smallmatrix}
\right ] $ with
$\Lambda _{r}=\mathsf{diag} \left \{ {\lambda _{1},\ldots ,\lambda _{r}}
\right \} $. Since the optimal likelihood ratio test for Gaussians with
distinct covariance matrices reduces to the QDA test, we apply it to
$\left (\Phi _{r}, \Psi _{r}\right )$ and get
\begin{equation}
T_{{\mathrm{eigen}}} = -\frac{1}{2}
\begin{bmatrix}
\Phi ^{\top}_{r} & \Psi ^{\top}_{r}
\end{bmatrix}
\left (\Sigma _{2r}^{-1} - \mathbf I_{2r} \right )
\begin{bmatrix}
\Phi _{r}
\\
\Psi _{r}
\end{bmatrix}
- \frac{1}{2} \log \det (\Sigma _{2r}).
\label{eq:Teigen}
\end{equation}
Since $(\Phi _{r},\Psi _{r})$ converges to
${\mathcal{N}}(\mathbf 0, \mathbf I_{2r})$ and the eigenvalues of
$\Sigma _{2r}^{-1} - \mathbf I_{2r}$ are
$\{\frac{-\sqrt{\alpha }\lambda _{k}}{1+\sqrt{\alpha }\lambda _{k}},
\frac{\sqrt{\alpha }\lambda _{k}}{1-\sqrt{\alpha }\lambda _{k}}\}_{k=1}^{r}$,
it follows that $T_{{\mathrm{eigen}}}$ converges in distribution to
$\xi _{r}$, where
\begin{equation*}
\xi _{r} = - \frac{1}{2}\sum _{k=1}^{r} \left (
\frac{\sqrt{\alpha}\lambda _{k}}{1-\sqrt{\alpha}\lambda _{k}} U_{k}^{2}
-\frac{\sqrt{\alpha}\lambda _{k}}{1+\sqrt{\alpha}\lambda _{k}} V_{k}^{2}
+ \log \left ( 1- \alpha \lambda _{k}^{2}\right ) \right ).
\end{equation*}
Interestingly, $\xi _{r}$ is exactly the first $r$ terms of the limiting
log likelihood ratio $\xi $ in~(\ref{eq:def_xi}). Therefore,
$T_{{\mathrm{eigen}}}$ achieves the optimal test power as $r \to \infty $ for
fixed ${\mathsf{P}}_{X,Y}$. Additionally, for the $d$-dimensional Gaussian
case, the likelihood ratio operator is the so-called Mehler kernel, whose
eigenfunctions are given by the Hermite polynomials (see~\prettyref{eq:mehler} In particular,
$\phi _{k}(x)= \psi _{k}(x) = x_{k}$ for $1 \le k \le d$,
$T_{{\mathrm{eigen}}}$ with $r=d$ reduces to the test in~(\ref{eq:LRT_sample_means})
based on sample means.

For the histogram test~(\ref{eq:histogram test}), recall that as
$n \to \infty $, $(s,t)$ converges in distribution to
${\mathcal{N}}(\mathbf 0, \Sigma _{0})$ and
${\mathcal{N}}(\mathbf 0, \Sigma _{1})$, under ${\mathcal{H}}_{0}$ and
${\mathcal{H}}_{1}$, where
$\Sigma _{0}=\left [
\begin{smallmatrix}
A & \mathbf 0
\\
\mathbf 0& B
\end{smallmatrix}
\right ]$ and
$\Sigma _{1}=\left [
\begin{smallmatrix}
A & \sqrt{\alpha}C
\\
\sqrt{\alpha}C^{\top }& B
\end{smallmatrix}
\right ]$ are two block matrices. Therefore, $T_{\mathrm{{hist}}}$ converges
in distribution to
\begin{equation*}
-\frac{1}{2} \sum _{k=1}^{w} \left (
\frac{\sqrt{\alpha }\mu _{k}}{1-\sqrt{\alpha }\mu _{k}} U_{k}^{2} -
\frac{\sqrt{\alpha }\mu _{k}}{1+\sqrt{\alpha }\mu _{k}}V_{k}^{2} +
\log \left ( 1-\alpha \mu _{k}^{2} \right ) \right ),
\end{equation*}
where $\{\mu _{k}\}_{k=1}^{w}$ are the singular values of matrix $C$, which
are also that of the discretized likelihood ratio kernel.

\begin{figure}[!ht]
    \centering
    \includegraphics[width=1\linewidth]
    {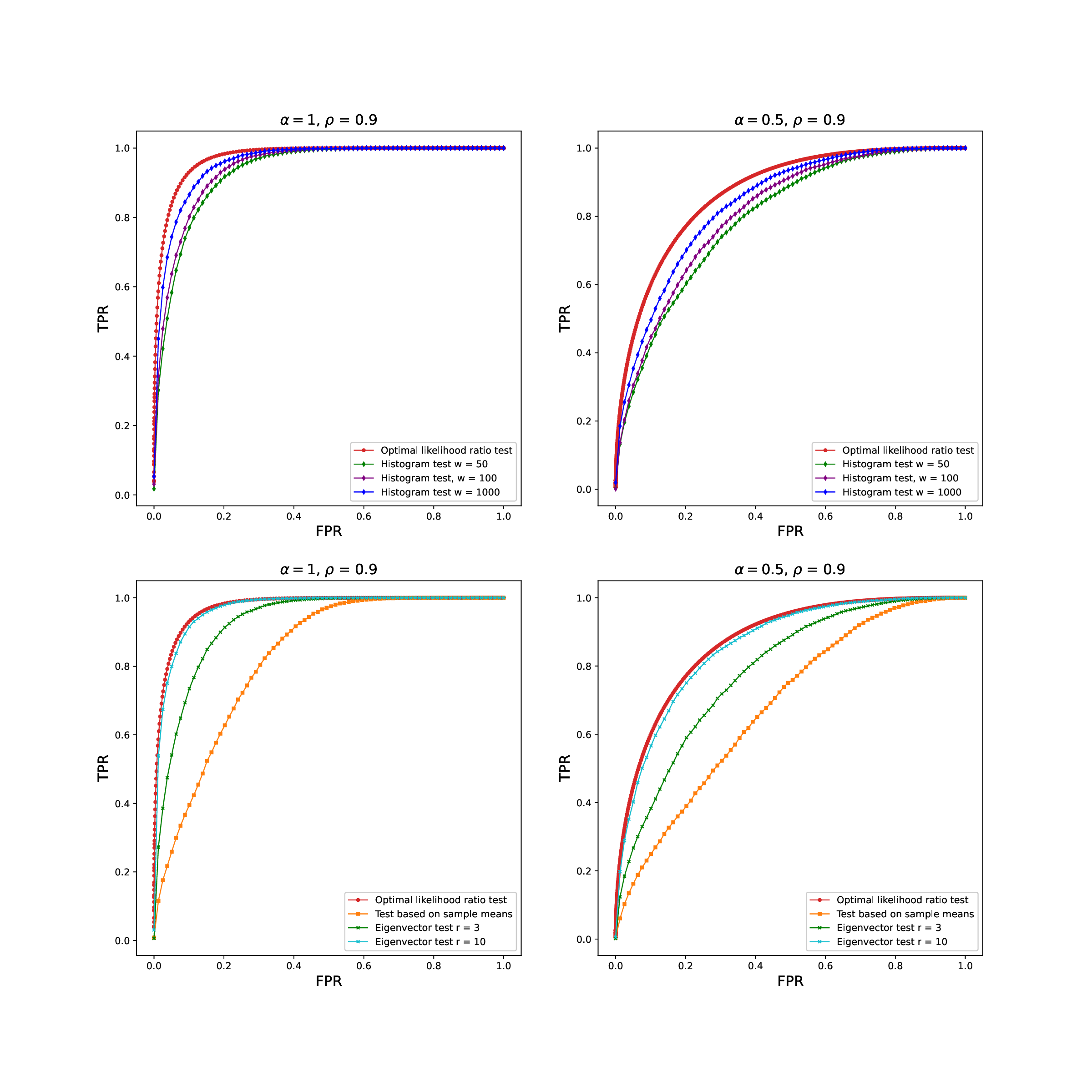}
        \caption{ROC curves of various tests for the broken bivariate Gaussian samples $(d=1)$.
    }
    \label{fig:ROC curve}

\end{figure}

\begin{figure}
    \centering
    \includegraphics[width=0.98\linewidth]{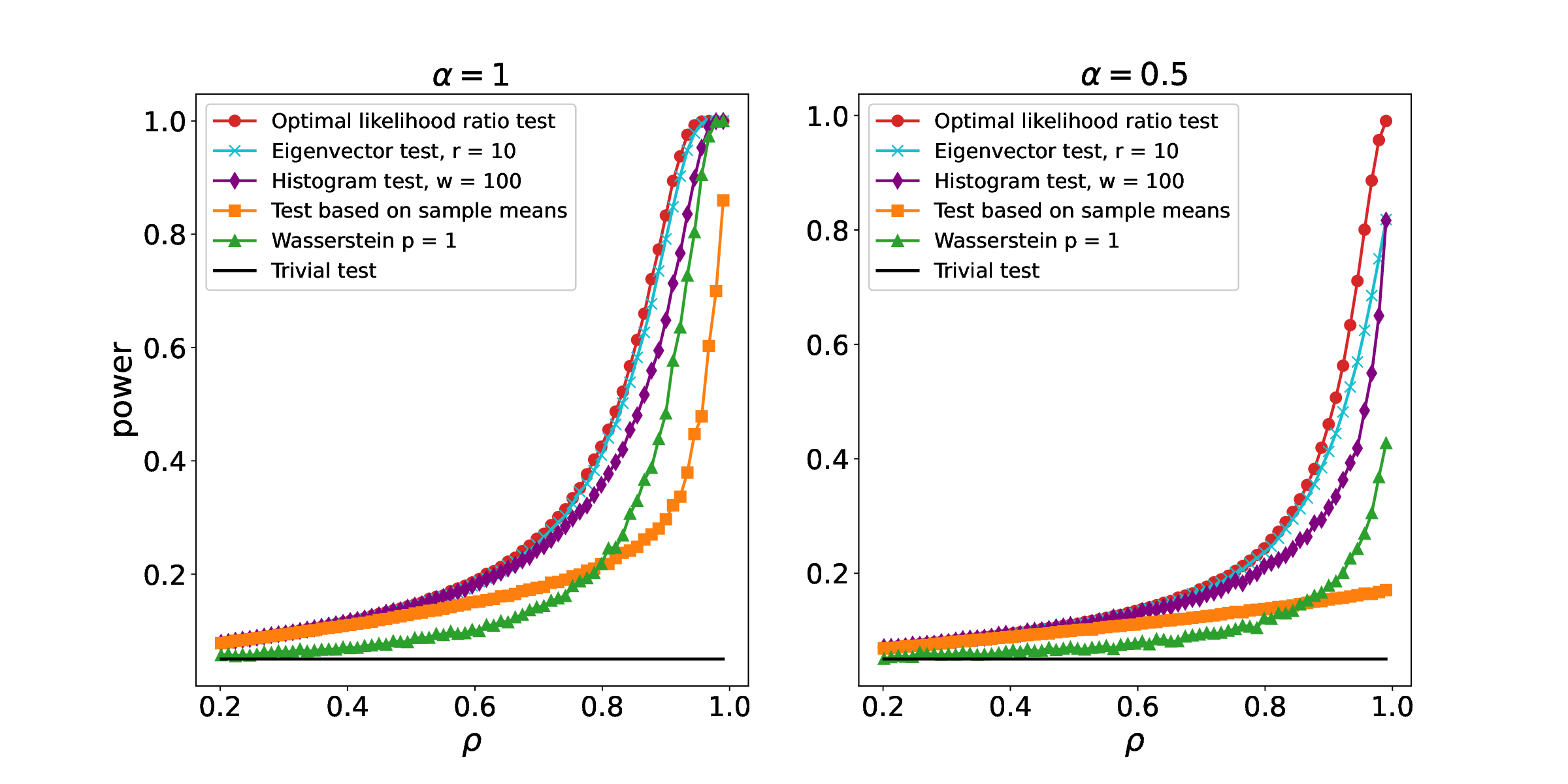}
    \caption{Power curves of various tests for the broken bivariate Gaussian samples $(d=1)$ and Type-I error  fixed at 0.05. Left: equal sample sizes $m=n$. 
    Right: unequal sample sizes $m=n/2$.}
    \label{fig:power curve}
\end{figure}

\subsubsection{Power analysis}
\label{sec3.4.2}

Next, we compare the asymptotic power ($n=\infty $) of these tests under
the Gaussian model in one dimension. The spectral embedding are obtained
by applying Hermite polynomials
$\phi _{k}=\psi _{k}= \frac{1}{\sqrt{k!}} H_{k}$ for degree up to
$r$. For histograms, we use $w$ quantile intervals for the Gaussian distribution
as the partition. Fixing $\rho =0.9$ and $\alpha = 1$ or $0.5$,
Fig.~\ref{fig:ROC curve} plots the receiver operating characteristic (ROC)
curves of the optimal likelihood ratio test, histogram tests for various
bin size $w$ (upper panel), and spectral test $T_{\mathrm{{eigen}}}$ for various
degree $r$ (lower panel). We make two observations:
\begin{itemize}
\item First, spectral embedding is much more efficient than histogram embedding,
in the sense that the power of the former converges rapidly to the optimal
curve as $r$ increases, whereas the latter requires a significantly larger
number of bins $w$ to achieve a comparable performance. This can be attributed
to exponential decay of eigenvalues in the Gaussian model (cf.~\prettyref{eq:mehler}) whereas quantization error
typically decays polynomially in the quantization level.
\item Second, unequal sample sizes are more difficult to test, as a constant
proportion of the $X$ sample (of unknown location) is independent of the
$Y$ sample, thereby further obscuring the correlation. Indeed, for both
spectral and histogram tests, smaller $\alpha $ requires a larger embedding
dimension ($r$ or $w$) to achieve a similar accuracy. For the spectral
test, the proof of Corollary~\ref{cor:Gaussian} suggests that $r$ needs to
be inversely proportional to $\alpha $ -- see~\prettyref{eq:r-choice}.
\end{itemize}

Based on these ROC curves, next we choose $w =100$ for histogram test and
$r = 10$ for spectral test and plot their power curves. For each test statistic,
we choose the threshold so that the Type-I error is fixed at 0.05 and plot
the power (one minus Type-II error) against the correlation level
$\rho \in [0.2,0.99]$. Also included are the test
(\ref{eq:LRT_sample_means}) based on sample means (which is a special
case of the spectral test for $r=1$) and the trivial test (rejecting
${\mathcal{H}}_{0}$ with probability $0.05$ independently of the data,
with a constant power of $0.05$.)

For equal sample size ($\alpha =1$, left panel of
Fig.~\ref{fig:power curve}), all tests have power approaching 1 as
$\rho \to 1$, including the test based on sample means. However, the performance
of degree-10 spectral test is much closer to optimal than that of histogram
with 100 bins. For unequal sample size ($\alpha =1/2$, right panel of
Fig.~\ref{fig:power curve}), all tests have lower power compared to the
case of $\alpha =1$. As remarked after Corollary~\ref{cor:Gaussian}, test based
solely on sample means fails to achieve power 1 even if $\rho \to 1$ and
it is necessary to incorporate higher-order eigenfunctions.

\subsubsection{Wasserstein test}
\label{sec3.4.3}

Finally, we discuss a natural test statistic based on the Wasserstein distance
for identical marginals ${\mathsf{P}}_{X}={\mathsf{P}}_{Y}$. As mentioned
in Section~\ref{sec:test and positive results}, the empirical distributions
$\hat{P}_{X}$ and $\hat{P}_{Y}$ are concentrated on the population, whose
fluctuations are independent under ${\mathcal{H}}_{0}$ and correlated under
${\mathcal{H}}_{1}$. As such, it is natural to expect that the typical
distance between $\hat{P}_{X}$ and $\hat{P}_{Y}$ tends to be smaller under
${\mathcal{H}}_{1}$ than that under ${\mathcal{H}}_{0}$. This motivates
the test which rejects ${\mathcal{H}}_{0}$ if
$W_{p}(\hat{P}_{X},\hat{P}_{Y})$, the $p$-Wasserstein distance between
these two empirical distributions, is smaller than a threshold. In the special
case of $n=m$, this amounts to solving a linear assignment problem:
\begin{equation}
W_{p}^{p}(\hat{P}_{X},\hat{P}_{Y}) = \frac{1}{n} \min _{\pi \in S_{n}}
\sum _{i=1}^{n} \| X_{\pi (i)} - Y_{i} \|^{p}.
\label{eq:wass}
\end{equation}
which, in the case of $p=2$, is the same as the MLE
(\ref{eq:MLE}). For $p=1$, by the Kantorovich-Rubinstein duality, it
can be interpreted as computing the maximum difference between the empirical
averages over all $1$-Lipschitz test functions (as opposed to explicit
test functions based on spectral or histogram embedding). The Wasserstein
test is appealing for two reasons: (a) it is \textit{universal} without
requiring knowledge of the joint distribution; (b) it scales well with
the ambient dimension $d$ and can be computed in time
$O(n^{2} d+n^{3})$ (first computing the weight matrix and solving the linear
assignment). In comparison, how the embedding dimension of histogram or
spectral embedding scales with $d$ must be chosen carefully depending on
the joint distribution. (For histogram, it will inevitably be exponential.)

The power curves for this test are also shown in~Fig.~\ref{fig:power curve}
for $n = 1000$ and $p = 1$ (The curves for $p=2$ are qualitatively similar
and omitted for brevity.) As seen in the left panel for equal sample sizes
($\alpha =1$), although for smaller $\rho $ the Wasserstein test underperforms
the sample mean test, its power eventually approaches $1$ as
$\rho \to 1$. However, its performance remains inferior to the histogram
and spectral tests, which more closely approximate the optimal likelihood
ratio test. For unequal sample sizes ($\alpha <1$), it is not hard to see
that the Wasserstein test fails to achieve strong detection\footnote{To
see this, note that under ${\mathcal{H}}_{0}$, as $n\to \infty $,
$\sqrt{n} W_{1}(\hat P_{X},\hat P_{Y})$ converges to some non-degenerate
law fully supported on ${\mathbb{R}}_{+}$ (cf.~\cite[Theorem 1.1(a1)]{del1999central}).
Under ${\mathcal{H}}_{1}$,
$\sqrt{n} W_{1}(\hat P_{X},\hat P_{Y}) = \Theta _{P}(1)$ by considering
linear test function. As such, the Wasserstein test cannot attain power
1 for unequal sample sizes.} even if $\rho = 1$ (as seen in the right panel
for $\alpha =0.5$).

\section{Proof of lower bounds}
\label{sec: proof of lower bounds}

\begin{proof}[Proof of~Theorem~\ref{thm:expression of second moment}]
\label{pf:proof_of_thm:expression_of_second_moment_VTEX1}
We first simplify the expression for
\begin{align}
\label{eq: second moment expansion}
\mathbb{E}_{0} {\mathbf{L}}^{2} & =
\frac{1}{\left | S_{m,n} \right |^{2}} \sum _{\pi ,\pi ' \in S_{m,n}}
\mathbb{E}_{0} \left [ \prod _{i=1}^{m} L(X_{\pi (i)},Y_{i}) L(X_{
\pi '(i)},Y_{i}) \right ]
\nonumber
\\
& = \frac{1}{\left | S_{m,n} \right |} \sum _{\pi \in S_{m,n}}
\mathbb{E}_{0} \left [ \prod _{i=1}^{m} L(X_{\pi (i)},Y_{i}) L(X_{i},Y_{i})
\right ],
\end{align}
where the second equality holds by symmetry and setting
$\pi '(i) =i,\forall i\in [m]$.

For a fixed $\pi $, a key observation is
\begin{align}
\mathbb{E}_{0} \left [ \prod _{i=1}^{m} L(X_{\pi (i)},Y_{i}) L(X_{i},Y_{i})
\right ]= \mathbb{E}_{0} \left [ \prod _{i \in I} L(X_{\pi (i)},Y_{i})
L(X_{i},Y_{i}) \right ],
\label{eq:two_core}
\end{align}
where $I $ is defined as the largest subset of $[m]$ such that
$\pi (I)=I$. To see why this equality holds, note that if $X_{j}$ appears
only once, say $L(X_{j}, Y_{i})$, in the product of
$\prod _{i=1}^{m} L(X_{\pi (i)},Y_{i})L(X_{i},Y_{i})$, then
$L(X_{j},Y_{i})$ can be eliminated by taking the expectation over
$X_{j}$ and using the fact that $\mathbb{E}_{X} L(X,y) = \int
\frac{{\mathrm{d}}P_{X,Y}}{{\mathrm{d}}P_{X} {\mathrm{d}}P_{Y}(y)} (x,y) {\mathrm{d}}P_{X}(x)
= \int{\mathrm{d}}P_{X \mid Y=y}(x) = 1$). Therefore, we can eliminate all
those $X_{j}$'s that appear only once. Similarly, if there exists
$Y_{i}$ appearing only once, say $L(X_{j},Y_{i})$, in the remaining product,
then $L(X_{j},Y_{i})$ can be also eliminated by taking the expectation
over $Y_{i}$ and using the fact that
$\mathbb{E}_{Y_{i}} [L(X_{j},Y_{i})] = 1$.

This iterative elimination procedure continues until no $X_{j}$ or
$Y_{i}$ can be dropped. The remaining indices are then given by
$I \subset [m]$, i.e., all $X_{i}$ and $Y_{i}$ appear exactly twice
for $i\in J$. Alternatively, one can view this iterative elimination procedure
as finding the $2$-core of a bipartite graph $G_{\pi}$ with left node set
$[m]$ and right node set $[n]$ and edge set
$\{(i,\pi (i))\}_{i=1}^{m} \cup \{(i,i)\}_{i=1}^{m}$. After successively
removing all nodes with degree at most $1$ together with its incident edge,
the remaining subgraph is the $2$-core of $G_{\pi}$ and the set of remaining
left and right nodes are both $I$. See Fig.~\ref{fig:2core} for an illustration.
The degree-$2$ constraint implies $\pi (I)=I$ and $I$ is the unique largest
subset with this property by the definition of $2$-core as the unique maximal
induced subgraph in which every vertex has degree at least $2$.
\begin{figure}[ht]
    \centering

\begin{tikzpicture}[
    scale=0.9,
    every node/.style={draw,circle,minimum size=0.7cm,inner sep=0pt}
  ]

  \node (L1) at (0,4) {1};
  \node (L2) at (0,3) {2};
  \node (L3) at (0,2) {3};
  \node (L4) at (0,1) {4};
  \node (L5) at (0,0) {5};
  \node (L6) at (0,-1) {6};

  \node (R1) at (3,4) {1};
  \node (R2) at (3,3) {2};
  \node (R3) at (3,2) {3};
  \node (R4) at (3,1) {4};
  \node (R5) at (3,0) {5};

  \node (R6) at (3,-1) {6};
  \node (R7) at (3,-2) {7};
  \node (R8) at (3,-3) {8};

  \draw[thick, dashed, blue] (L1) -- (R1);
  \draw[thick, dashed, blue] (L2) -- (R2);
  \draw[thick, dashed, blue] (L3) -- (R3);
  \draw[thick, dashed, blue] (L4) -- (R4);
  \draw[thick, dashed, blue] (L5) -- (R5);
  \draw[thick, dashed, blue] (L6) -- (R6);

  \draw[thick, red] (L1) -- (R2);
  \draw[thick, red] (L2) -- (R3);
  \draw[thick, red] (L3) -- (R1);
  \draw[thick, red] (L4) -- (R7);
  \draw[thick, red] (L5) -- (R4);
  \draw[thick, red] (L6) -- (R8);
\end{tikzpicture}
    \caption{An instance of the bipartite graph $G_\pi$ for $n=8$ and $m=6$. The edges $(i,i)$ and $(i,\pi(i))$ are in blue (dashed) and red (solid), respectively. Here, the 2-core is a 6-cycle with $I=\{1,2,3\}$. After removing this 2-core, the remaining graph consists two disjoint paths of lengths 2 and 4.}
    \label{fig:2core}
\end{figure}
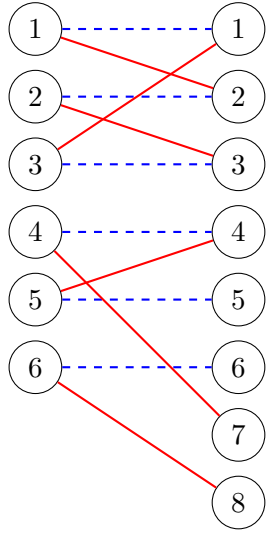

Note that $\pi $ restricted to $I$, $\pi _{|_{I}}$, is a permutation on
$I$. To further compute the RHS of~(\ref{eq:two_core}), let
${\mathcal{O}}$ denote the set of disjoint cycles of $\pi _{|_{I}}$ and
$N_{i}$ denote the number of cycles of length $i$. (Note that the 2-core
of $G_{\pi}$ must be a disjoint union of cycles; there are exactly
$N_{i}$ cycles of length $i$.) Recall that in~(\ref{eq:H-S operator}),
${\mathcal{L}}$ is defined as an integral operator from
$L_{2}({\mathsf{P}}_{Y})$ to $L_{2}({\mathsf{P}}_{X})$ with the kernel
$L(x,y)$. The adjoint operator ${\mathcal{L}}^{*}$ satisfies
$({\mathcal{L}}^{*} f)(y) = \int L(x,y)f(x) {\mathrm{d}}{\mathsf{P}}_{X}(x)$
for $f\in L_{2}({\mathsf{P}}_{X})$. Consequently, the composition
${\mathcal{T}}\coloneqq{\mathcal{L}}{\mathcal{L}}^{*}$ is a self-adjoint
operator on $L_{2}({\mathsf{P}}_{X})$ with the kernel given by
$T(x,x') = \int _{{\mathcal{Y}}}L(x,y)L(x',y) {\mathrm{d}}{\mathsf{P}}_{Y}(y)
= \mathbb{E}_{Y}L(x,Y)L(x',Y)$. With these notations, we have
\begin{align}
\label{eq: computation over I}
\mathbb{E}_{0} \left [ \prod _{i\in I}L(X_{\pi (i)},Y_{i}) L(X_{i},Y_{i})
\right ] & \overset{(a)}= \prod _{O\in {\mathcal{O}}} \mathbb{E}_{0}
\left [ \prod _{i\in O}L(X_{\pi (i)},Y_{i}) L(X_{i},Y_{i}) \right ]
\nonumber
\\
& \overset{(b)}= \prod _{O\in {\mathcal{O}}} \mathbb{E}_{X} \left [
\prod _{i\in O} \mathbb{E}_{Y} \left ( L(X_{\pi (i)},Y_{i}) L(X_{i},Y_{i})
\right ) \right ]
\nonumber
\\
& \overset{(c)}= \prod _{O\in {\mathcal{O}}}\mathbb{E}_{X} \left [
\prod _{i\in O} T(X_{\pi (i)}, X_{i}) \right ]
\nonumber
\\
& \overset{(d)}= \prod _{O\in {\mathcal{O}}} \mathsf{Tr}({\mathcal{T}}^{|O|})
\overset{(e)}= \prod _{O\in {\mathcal{O}}}\left ( \sum _{k=0}^{
\infty }\lambda _{k}^{2\left | O \right |} \right )= \prod _{i=1}^{
\left | I \right |} \left ( \sum _{k=0}^{\infty }\lambda _{k}^{2i}
\right )^{N_{i}},
\end{align}
where $(a)$ holds due to the independence among
$\{X_{i},Y_{i}\}_{i\in O}$ across disjoint orbits
$O\in{\mathcal{O}}$; $(b)$ holds by conditioning on $X$ and first taking
expectation over $Y$; $(c)$ follows from the definition of kernel
$T(x,x')$; $(d)$ utilizes the definition of trace of operator; and
$(e)$ holds because $\{\lambda _{k}^{2}\}_{k=0}^{\infty }$ are the eigenvalues
of ${\mathcal{T}}$.

Combining~(\ref{eq: second moment expansion}), (\ref{eq:two_core})
and~(\ref{eq: computation over I}) gives that
\begin{align}
\mathbb{E}_{0} {\mathbf{L}}^{2} = \frac{1}{\left | S_{m,n} \right |}
\sum _{\pi \in S_{m,n}} \prod _{i=1}^{\left | I \right |} \left (
\sum _{k=0}^{\infty }\lambda _{k}^{2i} \right )^{N_{i}}=
\frac{1}{\left | S_{m,n} \right |} \sum _{I \subset [m]} \sum _{
\sigma \in S_{I}} M(I,\sigma ) \prod _{i=1}^{\left | I \right |}
\left ( \sum _{k=0}^{\infty }\lambda _{k}^{2i} \right )^{N_{i}},
\label{eq:second_moment_cycle_decomp}
\end{align}
where $M(I,\sigma )$ denotes the number of permutations $\pi $ whose corresponding
bipartite graph $G_{\pi}$ has $2$-core given by $I$ and
$\pi _{|_{I}} = \sigma $.

To determine $M(I,\sigma )$, recall the bipartite graph formulation. Here
the edges in $\{(i,i)\}_{i\in [m]} \cup \{(i,\pi (i)) \}_{i \in [I]}$ have
already been fixed, and $I$ is the left/right node set of the $2$-core.
Therefore, to specify $\pi $, it is equivalent to adding
$m - \left | I \right |$ edges between $[m]\setminus I$ and
$[n] \setminus I$. Additionally, for $I$ being the $2$-core, after adding
those $m-\left | I \right |$ edges, the induced subgraph of
$G_{\pi}$ on $([m]\setminus I) \times ([n]\setminus I)$ must be a vertex-disjoint
collection of paths (cf.~Fig.~\ref{fig:2core}). We claim that the number
of possible ways to add those $m-\left | I \right |$ edges is
\begin{equation}
M(I,\sigma ) =
\begin{cases}
(n-\left | I \right | - 1) (n-\left | I \right | -2) \cdots (n-m),&
\left | I \right | < m
\\
1,&\left | I \right | = m,
\end{cases}
\label{eq:MIsigma}
\end{equation}
which only depends on the size of $I$; hence in the following, we denote
it by $M_{\left | I \right |}$. To justify this claim, we add edges one
by one according to an arbitrary ordering of the left vertices
$m\backslash I$. Suppose we have added $t$ edges incident to left vertices
$v_{1},\ldots ,v_{t}$. The next left vertex $v_{t+1}$ must be paired with
either a remaining degree-0 right vertex in $[n]\backslash [m]$ or one
of the remaining degree-1 right vertex in $[m]\backslash I$ that is
\textit{not} in the same path as $v_{t+1}$. The total number of such right
vertices is $n-|I|-t-1$. From here the claim (\ref{eq:MIsigma}) follows.

Finally, by classifying $I$ according to its size, we obtain that
\begin{align*}
\mathbb{E}_{0} {\mathbf{L}}^{2} & =
\frac{1}{\left | S_{m,n} \right |} \sum _{I \in [m]} \sum _{\sigma
\in S_{I}} M_{\left | I \right |} \prod _{i=1}^{\left | I \right |}
\left ( \sum _{k=0}^{\infty }\lambda _{k}^{2i} \right )^{N_{i}}
\\
& = \frac{1}{\left | S_{m,n} \right |} \sum _{\ell = 0}^{m}
\binom{m}{\ell} M_{\ell }\sum _{\sigma \in S_{\ell}}\prod _{i=1}^{
\left | I \right |} \left ( \sum _{k=0}^{\infty }\lambda _{k}^{2i}
\right )^{N_{i}}
\\
& \overset{(c)} = \frac{1}{\left | S_{m,n} \right |} \sum _{\ell = 0}^{m}
\binom{m}{\ell} \ell ! M_{\ell }a_{\ell} = \sum _{\ell = 0}^{m} t_{
\ell}a_{\ell},
\end{align*}
where $(c)$ follows from~Lemma~\ref{lmm:cycle index} and
$t_{\ell} = \frac{\binom{m}{\ell}\ell ! M_{\ell}}{\binom{n}{m}m!} $. Observe
that
\begin{equation*}
\sum _{\ell =0}^{m} t_{\ell} =
\frac{\sum _{\ell =0}^{m}\binom{n-\ell -1}{m-\ell}}{\binom{n}{m}} =1,
\end{equation*}
since the numerator counts the subsets of $[n]$ of size $m$ based on the
largest element $\ell $ for which the subset contains all elements in
$[\ell ]$.
\end{proof}

\begin{proof}[Proof of~Corollary~\ref{cor:lower bound general}]
Note that
\begin{equation*}
t_{\ell} =
\frac{m(m-1)\cdots (m-\ell + 1) (n-m)}{n(n-1)\cdots (n-\ell )} \leq
\left ( \frac{m}{n} \right )^{\ell }\frac{n-m}{n-\ell},
\text{ if $m < n$},
\end{equation*}
and when $m = n$, we have $t_{n} = 1$ and
$t_{\ell }= 0, \forall \ell < n$.

One can then obtain a upper bound for
$\mathbb{E}_{0} {\mathbf{L}}^{2}$:
\begin{equation*}
\mathbb{E}_{0} {\mathbf{L}}^{2} \leq \sum _{\ell = 0}^{m} \left (
\frac{m}{n} \right )^{\ell }\frac{n-m}{n-\ell} a_{\ell }\leq \sum _{
\ell = 0}^{\infty }\left ( \frac{m}{n} \right )^{\ell }a_{\ell }=
\prod _{k=0}^{\infty }\frac{1}{1 - (m/n) \lambda ^{2}_{k}},
\end{equation*}
where the second inequality follows from the nonnegativity of
$a_{\ell }\geq 0$ and the last equality applies the definition of
$a_{\ell}$.

Now, consider the following two cases separately:
\begin{itemize}
\item $m/n < 1-c_{1}$: We have
\begin{equation*}
\prod _{k=0}^{\infty }\frac{1}{1-m/n\lambda _{k}^{2}} = \exp \left (
\sum _{k=0}^{\infty }-\ln \left ( 1-\frac{m}{n}\lambda _{k}^{2}
\right ) \right ) \leq \exp \left ( \sum _{k=0}^{\infty }
\frac{\frac{m}{n}\lambda ^{2}_{k}}{1-\frac{m}{n}\lambda _{k}^{2}}
\right ) \leq \exp \left ( \frac{m}{c_{1}n} \sum _{k=0}^{\infty }
\lambda _{k}^{2} \right ),
\end{equation*}
which is $O(1)$ due to the assumption $\frac{m}{n}\sum _{k=1}^{
\infty }\lambda ^{2}_{k} \leq C_{1}$.
\item When $m=(1-o(1)) n$: Assume $\lambda _{1} < 1- c$ and
$\sum _{k=1}^{\infty }\lambda ^{2}_{k} <C$. If $m = n$, we have
$\sum _{\ell = 0}^{m} t_{\ell }a_{\ell }= a_{n}$ and the desired result
follows from~Lemma~\ref{lmm:limit of a_n}. In what follows, suppose
$m < n$.

We have
\begin{align*}
\left | \sum _{\ell =0}^{m} t_{\ell } a_{\ell }- \prod _{k=1}^{
\infty }(1-\lambda ^{2}_{k})^{-1} \right | & = \left | \sum _{\ell =0}^{m}
t_{\ell } \left ( a_{\ell }- \prod _{k=1}^{\infty }(1-\lambda ^{2}_{k})^{-1}
\right ) \right |\leq \sum _{\ell =0}^{m}C' t_{\ell} r^{-\ell}
\\
& \leq C'\sum _{\ell =0}^{m} \frac{n-m}{n-\ell} r^{-\ell} =o(1),
\end{align*}
where the first equality holds due to
$\sum _{\ell =0}^{m} t_{\ell} = 1$; the first inequality applies~Lemma~\ref{lmm:limit of a_n},
where $r>1$ and $C'>0$ are constants only depending on $c, C$; the second
inequality holds due to $t_{\ell }\le (n-m)/(n-\ell )$; the last inequality
follows because
$\sum _{\ell =0}^{m} \frac{n-m}{n-\ell} r^{-\ell} \leq
\frac{n-m}{n-m/2} \sum _{\ell \leq m/2} r^{-\ell} + \sum _{\ell \geq m/2}
r^{-\ell} \to 0$, using $n-m=o(n)$ and $m\to \infty $.
\end{itemize}
\end{proof}

\section{Proof of upper bounds}
\label{sec:pf-upper}

\begin{proof}[Proof of~Theorem~\ref{thm:upper bound general eigen 1}]
\label{pf:_proof_of_thm:upper_bound_general_eigen_1_VTEX1}
Consider the test statistic~(\ref{eq:statistic when lambda_1 to 1}),
namely
\begin{equation*}
T_{\mathrm{{top}}}= \left ( A-B \right )^{2}, \quad A\equiv
\frac{1}{\sqrt{n}} \sum _{i=1}^{n} \phi _{1}(X_{i}), \; B\equiv
\frac{1}{\sqrt{m}} \sum _{i=1}^{m} \psi _{1}(Y_{i}) .
\end{equation*}
Under ${\mathcal{H}}_{0}$, $X_{1},\cdots ,X_{n},Y_{1}\cdots ,Y_{m}$ are
all independent, so
\begin{align*}
\mathbb{P}_{0}\left \{ T_{\mathrm{{top}}}\leq \tau \right \} = \mathbb{E}_{B}[
\mathbb{P}\left \{ A \in [B-\sqrt{\tau },B+\sqrt{\tau }] \right \} ].
\end{align*}
It remains to argue that $A$ is anti-concentrated. Since
$\mathbb{E}_{{\mathsf{P}}_{X}} [\phi ^{2}_{1}(X)]=1$, by the Berry-Esseen
inequality, we have for some absolute constant $C$,
\begin{equation*}
\sup _{x\in {\mathbb{R}}} | \mathbb{P}\left \{ A \leq x \right \} -
\Phi (x)| \leq
\frac{C \mathbb{E}_{{\mathsf{P}}_{X}} |\phi ^{3}_{1}(X)|}{\sqrt{n}},
\end{equation*}
where $\Phi $ is the standard normal CDF. So
$\mathbb{P}_{0}\left \{ T_{\mathrm{{top}}}\leq \tau \right \} \leq C'(\sqrt{
\tau} +
\frac{ \mathbb{E}_{{\mathsf{P}}_{X}} |\phi ^{3}_{1}(X)|}{\sqrt{n}})$ for
some absolute constant $C'$.

Under ${\mathcal{H}}_{1}$, by symmetry, we can assume $\pi $ to be identity
and then
\begin{align*}
\mathbb{E}_{1} \left [ T_{\mathrm{{top}}} \right ] &= \mathbb{E}_{1} (A-B)^{2}
\\
& = \frac{1}{n} \mathbb{E}_{1} \left [ \sum _{i=1}^{n} \phi _{1}(X_{i})
\right ]^{2} + \frac{1}{m} \mathbb{E}_{1} \left [ \sum _{i=1}^{m}
\psi _{1}(Y_{i}) \right ]^{2} -\frac{2}{\sqrt{nm}} \sum _{i=1}^{n}
\sum _{j=1}^{m} \mathbb{E}_{1}\left [ \phi _{1}(X_{i})\psi _{1}(Y_{j})
\right ]
\\
& = 2 - \frac{2}{\sqrt{nm}}\sum _{i=1}^{m} \mathbb{E}_{{\mathsf{P}}_{X,Y}}
\left [ \phi _{1}(X) \psi _{1}(Y) \right ]= 2 - 2\sqrt{\frac{m}{n}}
\lambda _{1} \equiv 2\epsilon .
\end{align*}

By Markov inequality,
$\mathbb{P}_{1}\left \{ T_{\mathrm{{top}}}\geq \tau \right \} \leq
\frac{2\epsilon}{\tau}$. Choosing $\tau =\sqrt{\epsilon}$, we have
$\mathbb{P}_{1}\left \{ T_{\mathrm{{top}}}\geq \tau \right \} + \mathbb{P}_{0}
\left \{ T_{\mathrm{{top}}}\leq \tau \right \} \to 0$ provided that
$\epsilon \to 0$ and
$\mathbb{E}_{{\mathsf{P}}_{X}} |\phi ^{3}_{1}(X)| = o(\sqrt{n})$.

Under condition
$\mathbb{E}_{{\mathsf{P}}_{Y}} |\psi ^{3}_{1}(Y)| = o(\sqrt{m})$, similar
argument applies.
\end{proof}

\begin{proof}[Proof of~Theorem~\ref{thm:upper bound general eigen r}]
\label{pf:_proof_of_thm:upper_bound_general_eigen_r_VTEX1}
Consider the test statistic
(\ref{eq:m/n sum lambda^2_k goes to infinity}), namely
\begin{equation*}
T_{\mathrm{{inner}}}= \frac{1}{\sqrt{mn}} \sum _{i=1}^{n}\sum _{j=1}^{m} L_{r}(X_{i},Y_{j})
= \frac{1}{\sqrt{mn}} \sum _{i=1}^{n}\sum _{j=1}^{m} \sum _{k=1}^{r}
\lambda _{k} \phi _{k}(X_{i})\psi _{k}(Y_{j}),
\end{equation*}
and the test
\begin{equation*}
\eta _{r} =
\begin{cases}
1,& T_{\mathrm{{inner}}}>
\frac{\sqrt{\frac{m}{n}}\sum _{k=1}^{r} \lambda _{k}^{2}}{2};
\\
0,& \text{otherwise};
\end{cases}
.
\end{equation*}
Suppose that for a given choice of $r$ (which may depend on $n$), we have
\begin{equation*}
\lim _{n\to \infty} I_{\chi ^{2}}^{(r)}(X;Y) = \infty ,\quad
\mathbb{E}_{{\mathsf{P}}_{X,Y}} L^{2}_{r}(X,Y) = o\left ( n (I_{\chi ^{2}}^{(r)}(X;Y))^{2}
\right ).
\end{equation*}

We will first compute
$\mathbb{E}_{0} \left [ T_{\mathrm{{inner}}} \right ]$,
$\mathbb{E}_{1} \left [ T_{\mathrm{{inner}}} \right ]$,
$\mathsf{var}_{0} \left [ T_{\mathrm{{inner}}} \right ]$ and
$\mathsf{var}_{1}\left [ T_{\mathrm{{inner}}} \right ]$ and then use Chebyshev's
inequality to show that strong detection is achievable. First, by the orthonormality
of $\{\phi _{k}\}_{k\geq 0}$,
\begin{equation*}
\mathbb{E}_{0} \left [ T_{\mathrm{{inner}}} \right ] = \frac{1}{\sqrt{mn}}
\sum _{i=1}^{n}\sum _{j=1}^{m} \sum _{k=1}^{r} \lambda _{k}
\mathbb{E}_{0}[\phi _{k}(X_{i})\psi _{k}(Y_{j})] = 0.
\end{equation*}

Moreover, by symmetry, we may set $\pi $ to be identity and then,
\begin{equation*}
\mathbb{E}_{1} \left [ T_{\mathrm{{inner}}} \right ] = \frac{1}{\sqrt{mn}}
\sum _{i=1}^{m}\sum _{j=1}^{m} \sum _{k=1}^{r} \lambda _{k}
\mathbb{E}_{1}\phi _{k}(X_{i})\psi _{k}(Y_{j}) = \frac{1}{\sqrt{mn}}
\sum _{i=1}^{m}\sum _{k=1}^{r} \lambda _{k} \mathbb{E}_{1}\phi _{k}(X_{i})
\psi _{k}(Y_{i}) = \sqrt{\frac{m}{n}}\sum _{k=1}^{r} \lambda ^{2}_{k},
\end{equation*}
where the second equality holds because when $i \neq j$,
$(X_{i}, Y_{j})$ is distributed as
${\mathsf{P}}_{X} \otimes {\mathsf{P}}_{Y}$ and
$\mathbb{E}_{{\mathsf{P}}_{X}} \left [ \phi _{k}(X_{i}) \right ]=
\mathbb{E}_{{\mathsf{P}}_{Y}}\left [ \psi _{k}(Y_{j}) \right ] = 0$; the
last equality follows because $(X_{i},Y_{i})$ is distributed as
${\mathsf{P}}_{X,Y}$ so that
$\mathbb{E}_{1}\left [ \phi _{k}(X_{i})\psi _{k}(Y_{i}) \right ]=
\lambda _{k}$.

For $\mathsf{var}_{0} \left [ T_{\mathrm{{inner}}} \right ]$, we have
\begin{align*}
\mathsf{var}_{0} \left [ T_{\mathrm{{inner}}} \right ] & = \mathbb{E}_{0}
\left [ T_{\mathrm{{inner}}}^{2} \right ] = \frac{1}{nm} \sum _{i,i'=1}^{n}
\sum _{j,j'=1}^{m} \mathbb{E}_{0} \left [ L_{r}(X_{i},Y_{j})L_{r}(X_{i'},Y_{j'})
\right ]
\\
& =\frac{1}{nm}\sum _{i,i' = 1}^{n}\sum _{j,j'=1}^{m} \mathbb{E}_{0}
\left [ \sum _{k=1}^{r}\lambda _{k} \phi _{k}(X_{i})\psi _{k}(Y_{j})
\sum _{k'=1}^{r}\lambda _{k'} \phi _{k'}(X_{i'}) \psi _{k'}(Y_{j'})
\right ]
\\
& = \frac{1}{nm} \sum _{k,k'=1}^{r}\lambda _{k} \lambda _{k'} \sum _{i,i'=1}^{n}
\sum _{j,j'=1}^{m} \mathbb{E}\left [ \phi _{k}(X_{i})\phi _{k'}(X_{i'})
\right ] \mathbb{E}\left [ \psi _{k}(Y_{j}) \psi _{k'}(Y_{j'})
\right ]
\\
& = \frac{1}{nm} \sum _{k,k'=1}^{r}\lambda _{k} \lambda _{k'} \sum _{i=1}^{n}
\sum _{j=1}^{m} \delta _{k,k'} = \sum _{k=1}^{r}\lambda _{k}^{2}.
\end{align*}

For
$\mathsf{var}_{1} \left [ T_{\mathrm{{inner}}} \right ] = \mathbb{E}_{1}
\left [ T_{\mathrm{{inner}}}^{2} \right ] - (\mathbb{E}_{1} \left [ T_{\mathrm{{inner}}}
\right ])^{2}$, we only need to calculate
$\mathbb{E}_{1} \left [ T_{\mathrm{{inner}}}^{2} \right ]$, which by symmetry
is
\begin{equation*}
\mathbb{E}_{1} \left [ T_{\mathrm{{inner}}}^{2} \right ] = \frac{1}{nm}
\sum _{i,i' = 1}^{n} \sum _{j,j'=1}^{m} \mathbb{E}_{1|{\mathrm{Id}}}\left [
L_{r}(X_{i},Y_{j})L_{r}(X_{i'},Y_{j'}) \right ].
\end{equation*}
The summands can be divided into five cases:

\begin{itemize}
\item $i=j=i'=j'$: We have
\begin{equation*}
\mathbb{E}_{1|\mathrm{Id}} \left [ L^{2}_{r}(X_{i},Y_{i}) \right ] =
\mathbb{E}_{{\mathsf{P}}_{X,Y}} \left [ L^{2}_{r}(X,Y) \right ].
\end{equation*}
\item $i=i',j=j'$ but $i\neq j$:
\begin{equation*}
\mathbb{E}_{1|\mathrm{Id}} \left [ L_{r}(X_{i},Y_{j})L_{r}(X_{i},Y_{j})
\right ] = \mathbb{E}_{{\mathsf{P}}_{X}\otimes {\mathsf{P}}_{Y}}
\left [ L^{2}_{r}(X_{i},Y_{j}) \right ] = \sum _{k=1}^{r} \lambda _{k}^{2}.
\end{equation*}
\item $i=j,i'=j'$ but $i\neq i'$:
\begin{align*}
\mathbb{E}_{1|\mathrm{Id}} \left [ L_{r}(X_{i},Y_{i})L_{r}(X_{i'},Y_{i'})
\right ] & = \left ( \mathbb{E}_{{\mathsf{P}}_{X,Y}} \left [ L_{r}(X,Y)
\right ] \right )^{2}
\\
&= \left ( \sum _{k=1}^{r} \lambda _{k}\mathbb{E}_{{\mathsf{P}}_{X,Y}}
\left [ \phi _{k}(X)\psi _{k}(Y) \right ] \right )^{2}=\left ( \sum _{k=1}^{r}
\lambda _{k}^{2} \right )^{2}.
\end{align*}
\item $i=j',i'=j$ but $i \neq i'$:
\begin{align*}
\mathbb{E}_{1|\mathrm{Id}} \left [ L_{r}(X_{i},Y_{j}) L_{r}(X_{j},Y_{i})
\right ] &= \sum _{k,k'=1}^{r} \lambda _{k} \lambda _{k'}\mathbb{E}_{{
\mathsf{P}}_{X,Y}} \left [ \phi _{k}(X_{i}) \psi _{k'}(Y_{i}) \right ]
\mathbb{E}_{{\mathsf{P}}_{X,Y}} \left [ \phi _{k'}(X_{j})\psi _{k}(Y_{j})
\right ]
\\
&= \sum _{k=1}^{r} \lambda _{k}^{4}.
\end{align*}
\item In the remaining case, there is at least one isolated index, hence
the corresponding summand must be $0$.
\end{itemize}
We conclude that
\begin{equation*}
\mathbb{E}_{1} \left [ T_{\mathrm{{inner}}}^{2} \right ] = \frac{1}{n}
\mathbb{E}_{{\mathsf{P}}_{X,Y}} \left [ L^{2}_{r}(X,Y) \right ] +
\frac{n-1}{n} \left ( \sum _{k=1}^{r} \lambda _{k}^{2} \right ) +
\frac{m-1}{n} \left ( \sum _{k=1}^{r} \lambda _{k}^{2} \right )^{2} +
\frac{m-1}{n} \left ( \sum _{k=1}^{r} \lambda _{k}^{4} \right ).
\end{equation*}
Therefore,
\begin{align*}
\mathsf{var}_{1} \left [ T_{\mathrm{{inner}}} \right ] & = \frac{1}{n}
\mathbb{E}_{{\mathsf{P}}_{X,Y}} \left [ L^{2}_{r}(X,Y) \right ]+
\frac{n-1}{n} \left ( \sum _{k=1}^{r} \lambda _{k}^{2} \right )-
\frac{1}{n} \left ( \sum _{k=1}^{r} \lambda _{k}^{2} \right )^{2} +
\frac{m-1}{n} \left ( \sum _{k=1}^{r} \lambda _{k}^{4} \right )
\\
& \le \frac{1}{n}\mathbb{E}_{{\mathsf{P}}_{X,Y}} \left [ L^{2}_{r}(X,Y)
\right ] + 2 \sum _{k=1}^{r} \lambda _{k}^{2},
\end{align*}
where the last inequality holds due to
$\lambda _{k}^{4} \leq \lambda _{k}^{2}$ for all $k\geq 1$ and
$m\leq n$.

Then the testing error of $\eta _{r}$ could be bounded from above by
\begin{align*}
&\mathbb{P}_{0} \left \{ T_{\mathrm{{inner}}}>
\frac{\sqrt{\frac{m}{n}}\sum _{k=1}^{r} \lambda _{k}^{2}}{2} \right
\} + \mathbb{P}_{1} \left \{ T_{\mathrm{{inner}}}\leq
\frac{\sqrt{\frac{m}{n}}\sum _{k=1}^{r} \lambda _{k}^{2}}{2} \right
\}
\\
& \leq
\frac{4\mathsf{var}_{0} \left [ T_{\mathrm{{inner}}} \right ]}{\frac{m}{n}(I_{\chi ^{2}}^{(r)}(X;Y))^{2}}
+
\frac{4\mathsf{var}_{1} \left [ T_{\mathrm{{inner}}} \right ]}{\frac{m}{n}(I_{\chi ^{2}}^{(r)}(X;Y))^{2}}
\\
&\leq \frac{4}{\frac{m}{n}I_{\chi ^{2}}^{(r)}(X;Y)} +
\frac{4\mathbb{E}_{{\mathsf{P}}_{X,Y}} \left [ L^{2}_{r}(X,Y) \right ]}{m\left ( I_{\chi ^{2}}^{(r)}(X;Y) \right )^{2}}+
\frac{8}{\frac{m}{n}I_{\chi ^{2}}^{(r)}(X;Y)}\to 0,
\end{align*}
where the first inequality follows from Chebyshev's inequality and
$I_{\chi ^{2}}^{(r)}(X;Y) = \sum _{k=1}^{r} \lambda _{k}^{2}$.
\end{proof}

\begin{proof}[Proof of~Theorem~\ref{thm:fixed P Q}]
\label{pf:_proof_of_thm:fixed_P_Q_VTEX1}
The necessity part directly follows from~Corollary~\ref{cor:lower bound general}.

For sufficiency, we split the proof into two parts. First, assuming
$I_{\chi ^{2}}(X;Y) = \infty $, we show that strong detection is achievable
for any $0<\alpha \leq 1$. Since the $\chi ^{2}$-information is infinite,
the likelihood ratio operator is not necessarily Hilbert-Schmidt and may
not admit an eigen-decomposition as in~(\ref{eq:SVD of L}). To overcome
this, we discretize the space using finite partitions and construct histogram
statistics. By the Central Limit Theorem, these statistics are approximately
Gaussian. Since the histogram counts are linearly dependent (summing to
fixed values), leading to degenerate covariance matrices, we project the
statistics onto the space spanned by eigenvectors of covariance matrix
corresponding to non-zero eigenvalues. This projection ensures the covariance
matrices are invertible (under $\mathcal{H}_{0}$), allowing us to apply
the optimal test (QDA) for Gaussian distributions. Finally, by choosing
a sufficiently fine partition, the $\chi ^{2}$-divergence of the discretized
distributions can be made arbitrarily large, thereby driving the testing
error to zero.

For any fixed $K > 0$, by~Lemma~\ref{lmm:property of f-divergence on product space},
there exists a partition
${\mathcal{I}}= \{I_{1},I_{2},\cdots ,I_{w}\}$ on ${\mathcal{X}}$ and
${\mathcal{J}}= \{J_{1},J_{2},\cdots ,J_{w}\}$ on ${\mathcal{Y}}$ such
that the induced partition
${\mathcal{E}}= \{I_{k}\times J_{l}\}_{k,l = 1}^{w}$ on
${\mathcal{X}}\times {\mathcal{Y}}$ satisfies the following: if we define
the discretized joint distributions
${\mathsf{P}}_{{\mathcal{E}}}(k,l) = {\mathsf{P}}_{X,Y}(I_{k}\times J_{l})$
and
${\mathsf{Q}}_{{\mathcal{E}}}(k,l) = {\mathsf{P}}_{X}(I_{k}) {
\mathsf{P}}_{Y}(J_{l})$, then
$\chi ^{2}({\mathsf{P}}_{{\mathcal{E}}}\|{\mathsf{Q}}_{{\mathcal{E}}})
> K$.

Recall the definition of the standardized histograms $s$ and $t$ from
(\ref{eq:histogram vectors}), namely
\begin{equation*}
s_{k} = \frac{1}{\sqrt{n}} \sum _{i=1}^{n}
\frac{{\mathbf{1}_{\left \{{X_{i} \in I_{k}}\right \}}} - r_{k}}{ \sqrt{r_{k}}},
\quad t_{l}= \frac{1}{\sqrt{m}} \sum _{i=1}^{m}
\frac{{\mathbf{1}_{\left \{{Y_{i} \in J_{l}}\right \}}} - z_{l}}{\sqrt{z_{l}}},
\end{equation*}
where $r_{k} = {\mathsf{P}}_{X}(I_{k})$ for $k\in [w]$ and
$z_{l} = {\mathsf{P}}_{Y}(J_{l})$ for $l\in [w]$. Under the two hypotheses,
the covariance matrices of $(s^{\top},t^{\top})^{\top}$ can be computed
as follows:
\begin{equation*}
\Sigma _{0} =\mathbb{E}_{0}
\begin{bmatrix}
ss^{\top }& {st^{\top }}
\\
{ts^{\top }}& tt^{\top }
\end{bmatrix}
=
\begin{bmatrix}
A&\mathbf 0
\\
\mathbf 0&B
\end{bmatrix}
, \quad \Sigma _{1} = \mathbb{E}_{1}
\begin{bmatrix}
ss^{\top }& {st^{\top }}
\\
{ts^{\top }}& tt^{\top }
\end{bmatrix}
=
\begin{bmatrix}
A& \sqrt{\alpha} C
\\
\sqrt{\alpha} C^{\top}&B
\end{bmatrix}
,
\end{equation*}
where
$A = \mathbf I_{w} - \gamma \gamma ^{\top},B = \mathbf I_{w} - \beta
\beta ^{\top},C ={\widetilde{P}}-\gamma \beta ^{\top}$. Here
${\widetilde{P}}_{k,l} \equiv {\mathsf{P}}_{X,Y}( I_{k}\times J_{l})/
\sqrt{r_{k}z_{l}}$ and
$\gamma = (\sqrt{r_{1}},\sqrt{r_{2}},\cdots ,\sqrt{r_{w}})^{\top}$ and
$\beta = (\sqrt{z_{1}},\sqrt{z_{2}},\cdots ,\sqrt{z_{w}})^{\top}$ are both
unit vectors. To see this, note that,
\begin{align*}
A_{kl} & = \mathbb{E}_{0} s_{k} s_{l} = \frac{1}{n} \sum _{i=1}^{n}
\sum _{j=1}^{n} \mathbb{E}_{0} \left (
\frac{{\mathbf{1}_{\left \{{X_{i}\in I_{k}}\right \}}} - r_{k}} {\sqrt{r_{k}}}
\right ) \left (
\frac{{\mathbf{1}_{\left \{{X_{j} \in I_{\ell }}\right \}}} - r_{\ell }}{\sqrt{r_{\ell }}}
\right )
\\
&= \frac{1}{n} \sum _{i=1}^{n} \mathbb{E}_{0} \left (
\frac{{\mathbf{1}_{\left \{{X_{i}\in I_{k}}\right \}}} - r_{k}} {\sqrt{r_{k}}}
\right ) \left (
\frac{{\mathbf{1}_{\left \{{X_{i} \in I_{\ell }}\right \}}} - r_{\ell }}{\sqrt{r_{\ell }}}
\right )=
\begin{cases}
1-r_{k},&k = \ell
\\
-\sqrt{r_{k}r_{\ell }},&k \neq \ell
\end{cases},
\end{align*}
where the penultimate equality utilizes that $X_{i}$ and $X_{j}$ are independent
when $i\neq j$. The expressions for $B$ and $C$ can be computed in a similar
manner.

Note that both $\Sigma _{0}$ and $\Sigma _{1}$ have zero eigenvalues because
the entries of $s$ and $t$ are linearly dependent. To resolve this rank
deficiency, we first project $(s,t)$ to the space spanned by eigenvectors
of $\Sigma _{0}$ corresponding to non-zero eigenvalues.

Importantly, the left (resp.~right) singular vectors of the off-diagonal
block $C$ coincide with those of the diagonal block $A$ (resp.~$B$). Indeed,
we have
$\gamma ^{\top }A = \gamma ^{\top }- \gamma ^{\top }\gamma \gamma ^{
\top }= \mathbf 0$ and
$\gamma ^{\top }C = \gamma ^{\top }{\widetilde{P}}- \beta ^{\top }=
\mathbf 0$ since
\begin{equation*}
\sum _{j=1}^{w} \sqrt{r_{j}}
\frac{{\mathsf{P}}_{X,Y}(I_{j}\times J_{\ell })}{\sqrt{r_{j} z_{\ell }}}
=
\frac{\sum _{j=1}^{w} {\mathsf{P}}_{X,Y}(I_{j}\times J_{\ell })}{\sqrt{z_{\ell }}}
= \frac{{\mathsf{P}}_{Y}(J_{\ell })}{\sqrt{z_{\ell }}} = \sqrt{z_{
\ell }} = \beta _{\ell },
\end{equation*}
meaning that the $\ell $-th entries of
$\gamma ^{\top }{\widetilde{P}}$ and $\beta $ agree. Similarly, one can
verify that $B\beta = C\beta = \mathbf 0$. Assume the singular value decomposition
of $C$ is $C = U \Lambda V^{\top}$, where $U,V$ are orthogonal matrices
in the form of $U= ({\widetilde{U}}, \gamma )$ and
$V = ({\widetilde{V}},\beta )$ and
$\Lambda = \mathsf{diag} \left \{ {\mu _{1},\mu _{2},\cdots ,\mu _{w}}
\right \} $ with
$1\geq \mu _{1} \geq \mu _{2} \geq \cdots \geq \mu _{w} = 0$. (These are
in fact singular values of the discretized likelihood operator.) Then
we have
\begin{equation*}
{\widetilde{U}}^{\top }A{\widetilde{U}}= {\widetilde{U}}^{\top }{
\widetilde{U}}- {\widetilde{U}}^{\top }\gamma \gamma ^{\top }{
\widetilde{U}}= {\widetilde{U}}^{\top }{\widetilde{U}}= \mathbf I_{w-1},
\end{equation*}
where we use that the columns of ${\widetilde{U}}$ are orthonormal and
orthogonal to $\gamma $. Similarly,
${\widetilde{V}}^{\top }B{\widetilde{V}}= \mathbf I_{w-1}$ and
\begin{equation*}
{\widetilde{U}}^{\top }C{\widetilde{V}}= {\widetilde{U}}^{\top }({
\widetilde{U}},\gamma ) \Lambda ({\widetilde{V}},\beta )^{\top }{
\widetilde{V}}= (\mathbf I_{w-1},\mathbf 0) \Lambda (\mathbf I_{w-1},
\mathbf 0)^{\top }= \widetilde{\Lambda },
\end{equation*}
where diagonal matrix
$\widetilde{\Lambda } = \mathsf{diag} \left \{ {\mu _{1},\cdots ,\mu _{w-1}}
\right \} $. Let $\tilde{s} = {\widetilde{U}}^{\top } s$ and
$\tilde{t} = {\widetilde{V}}^{\top } t$, whose covariance matrices are
\begin{equation*}
\widetilde{\Sigma }_{0} = \mathbb{E}_{0}
\begin{bmatrix}
\tilde{s} \tilde{s}^{\top }&\tilde{s}\tilde{t}^{\top }
\\
\tilde{t}\tilde{s}^{\top }&\tilde{t}\tilde{t}^{\top }
\end{bmatrix}
=
\begin{bmatrix}
{\widetilde{U}}^{\top }&\mathbf 0
\\
\mathbf 0&{\widetilde{V}}^{\top }
\end{bmatrix}
\begin{bmatrix}
A &\mathbf 0
\\
\mathbf 0&B
\end{bmatrix}
\begin{bmatrix}
{\widetilde{U}}&\mathbf 0
\\
\mathbf 0&{\widetilde{V}}
\end{bmatrix}
= \mathbf I_{2w-2},
\end{equation*}
and
\begin{equation*}
\widetilde{\Sigma }_{1} = \mathbb{E}_{1}
\begin{bmatrix}
\tilde{s} \tilde{s}^{\top }&\tilde{s}\tilde{t}^{\top }
\\
\tilde{t}\tilde{s}^{\top }&\tilde{t}\tilde{t}^{\top }
\end{bmatrix}
=
\begin{bmatrix}
{\widetilde{U}}^{\top }&\mathbf 0
\\
\mathbf 0&{\widetilde{V}}^{\top }
\end{bmatrix}
\begin{bmatrix}
A &\sqrt{\alpha }C
\\
\sqrt{\alpha }C^{\top }&B
\end{bmatrix}
\begin{bmatrix}
{\widetilde{U}}&\mathbf 0
\\
\mathbf 0&{\widetilde{V}}
\end{bmatrix}
=
\begin{bmatrix}
\mathbf I_{w-1} &{\sqrt{\alpha }}\widetilde{\Lambda }
\\
{\sqrt{\alpha }}\widetilde{\Lambda }^{\top } &\mathbf I_{w-1}
\end{bmatrix}
.
\end{equation*}
Now, we have mapped the original data into $\tilde{s}$ and
$\tilde{t}$ whose covariance matrix under ${\mathcal{H}}_{0}$ is nondegenerate.

Consider two cases according to whether $\widetilde{\Sigma }_{1}$, the
covariance matrix under ${\mathcal{H}}_{1}$ is degenerate:

\underline{Case I}: $\mu _{1}=1$ and $\alpha =1$. In this case, under
${\mathcal{H}}_{1}$, $\text{Cov}(\tilde s,\tilde t)$ equals the first element
of $\sqrt{\alpha }\widetilde{\Lambda }$, which is $1$. Additionally, since
$\tilde s$ and $\tilde t$ have variance $1$ and mean $0$, they are equal
with probability 1. Under ${\mathcal{H}}_{0}$, by CLT,
$\tilde s_{1}-\tilde t_{1} \to {\mathcal{N}}(0,2)$ weakly as
$n\to \infty $. Thus the test
${\mathbf{1}_{\left \{{\tilde s_{1}=\tilde t_{1}}\right \}}}$ attains vanishing
probability of error.

\underline{Case II}: $\mu _{1}<1$ or $\alpha < 1$. As a result,
$\widetilde{\Sigma}_{1}$ is invertible, with eigenvalues
$1 \pm \sqrt{\alpha} \mu _{k}$. By CLT, as $n \to \infty $,
$(\tilde{s},\tilde{t})^{\top}$ converges to
${\mathsf{N}}_{0} \triangleq {\mathcal{N}}(\mathbf 0,
\widetilde{\Sigma}_{0})$ and
${\mathsf{N}}_{1} \triangleq {\mathcal{N}}(\mathbf 0,
\widetilde{\Sigma}_{1})$ under ${\mathcal{H}}_{0}$ and
${\mathcal{H}}_{1}$, respectively. Consider the likelihood ratio test
${\mathbf{1}_{\left \{{T>0}\right \}}}$ under the normal limits, where
\begin{equation*}
T = -
\begin{pmatrix}
\tilde{s}
\\
\tilde{t}
\end{pmatrix}
^{\top }\left ( \widetilde{\Sigma }_{1}^{-1}-\mathbf I_{2w-2} \right )
\begin{pmatrix}
\tilde{s}
\\
\tilde{t}
\end{pmatrix}
- \sum _{k=1}^{w} \log (1-\alpha \mu _{k}^{2}).
\end{equation*}
By weak convergence,
$\mathbb{P}_{0}[T>0] = \mathbb{P}_{{\mathsf{N}}_{0}}[T>0] + o_{n}(1)$ and
$\mathbb{P}_{1}[T\leq 0] = \mathbb{P}_{{\mathsf{N}}_{1}}[T\leq 0] + o_{n}(1)$.
Thus it remains to bound
$\mathbb{P}_{{\mathsf{N}}_{0}}[T>0]+\mathbb{P}_{{\mathsf{N}}_{1}}[T
\leq 0] = 1-\mathrm{TV}({\mathsf{N}}_{0},{\mathsf{N}}_{1})$.

To this end, recall that $\mu _{k}$ is the singular value of $C$. Consequently,
\begin{equation*}
\sum _{k=1}^{w}\mu _{k}^{2} =\mathsf{Tr}\left ( C^{\top }C \right ) =
\mathsf{Tr}\left ( {\widetilde{P}}^{\top }{\widetilde{P}}- {
\widetilde{P}}^{\top } \gamma \beta ^{\top } - \beta \gamma ^{\top }{
\widetilde{P}}+ \beta \gamma ^{\top }\gamma \beta ^{\top } \right )=
\mathsf{Tr}({\widetilde{P}}^{\top} {\widetilde{P}}) - 1,
\end{equation*}
where the last equality utilizes the fact that
$\mathsf{Tr}(\beta \gamma ^{\top }{\widetilde{P}}) = \mathsf{Tr}({
\widetilde{P}}\beta \gamma ^{\top }) = \mathsf{Tr}(\gamma \gamma ^{
\top }) = 1$ and
$C\beta = {\widetilde{P}}\beta - \gamma = \mathbf 0$.

By definition of ${\mathsf{P}}_{{\mathcal{E}}}$ and
${\mathsf{Q}}_{{\mathcal{E}}}$, we have
\begin{equation*}
\mathsf{Tr}\left ( {\widetilde{P}}^{\top }{\widetilde{P}} \right ) =
\sum _{k=1}^{w} \sum _{l=1}^{w}
\frac{{\mathsf{P}}_{X,Y}^{2}(I_{k}\times J_{l})}{r_{k}z_{l}} = \sum _{k,l}
\frac{{\mathsf{P}}_{X,Y}^{2}(I_{k}\times J_{l})}{{\mathsf{P}}_{X}\otimes {\mathsf{P}}_{Y}(I_{k}\times J_{l})}
= \sum _{k,l}
\frac{{\mathsf{P}}^{2}_{{\mathcal{E}}}(k,l)}{{\mathsf{Q}}_{{\mathcal{E}}}(k,l)}
= \chi ^{2}({\mathsf{P}}_{{\mathcal{E}}}\|{\mathsf{Q}}_{{\mathcal{E}}})+1.
\end{equation*}
Then
$\sum _{k=1}^{w-1} \mu _{k}^{2} = \chi ^{2}({\mathsf{P}}_{{
\mathcal{E}}}\|{\mathsf{Q}}_{{\mathcal{E}}}) > K$ by choice of the partition.
The total variation can then be bounded via Hellinger distance as follows:
\begin{align}
\label{eq:bound_total_variation_VTEX1}
1 - \mathrm{TV}({\mathsf{N}}_{0},{\mathsf{N}}_{1}) \leq 1 -
\frac{H^{2}({\mathsf{N}}_{0},{\mathsf{N}}_{1})}{2} =
\frac{{\mathrm{det}} \left ( \widetilde{\Sigma }_{0} \widetilde{\Sigma }_{1} \right )^{1/4}}{{\mathrm{det}} \left ( \frac{\widetilde{\Sigma }_{0} + \widetilde{\Sigma }_{1}}{2} \right )^{1/2}}
& =\left \{ \prod _{k=1}^{w-1}
\frac{1-\alpha \mu _{k}^{2}}{(1-\alpha \mu _{k}^{2}/4)^{2}} \right \}^{1/4}
\nonumber
\\
& \leq \exp \left ( -\frac{\alpha }{8} \sum _{k=1}^{w-1} \mu _{k}^{2}
\right ) \leq \exp \left ( -\frac{\alpha K}{8} \right ),
\end{align}
where the first equality applies the Hellinger distance between two Gaussians
(\cite[Sec.~7.7]{polyanskiy2025information}) and the second inequality
holds due to $(1-x)/(1-x/4)^{2} \le 1-x/2 \le \exp (-x/2)$. The proof is
complete by the arbitrariness of $K$.

Finally, assume $\alpha = 1$, $I_{\chi ^{2}}(X;Y) < \infty $ and
$\rho (X;Y) =\lambda _{1} = 1$. Consider the same test used in proving~Theorem~\ref{thm:upper bound general eigen 1}.
Note that ${\mathsf{P}}_{X,Y}$ is a fixed distribution. Hence, the third-moment
condition in~Theorem~\ref{thm:upper bound general eigen 1} automatically
holds, and $T_{\mathrm{{top}}}$ achieves strong detection when
$\rho (X;Y) = 1$.
\end{proof}

\section{Discussion and open problems}
\label{sec: discussion}

In this paper, we address the problem of testing dependency between two
datasets with missing correspondence. For a fixed joint distribution
${\mathsf{P}}_{X,Y}$ and the sample sizes $m =\alpha n$ for a constant
$\alpha \in (0,1]$, we establish that strong detection is possible if and
only if the $\chi ^{2}$-information
$I_{\chi ^{2}}(X;Y)=\chi ^{2}({\mathsf{P}}_{X,Y}\|{\mathsf{P}}_{X}
\otimes {\mathsf{P}}_{Y})$ is infinite, or the maximal correlation
$\rho (X;Y)$ equals $1$ and $\alpha =1$, thereby resolving Bai-Hsing's
conjecture positively. We further extend the results to more general settings
where ${\mathsf{P}}_{X,Y}$ or the ambient dimension $d$ may depend on
$n$. These general results determine the sharp detection thresholds for
Gaussian or Bernoulli models in both low and high dimensions, closing gaps
in the existing literature.

We close the paper by discussing several open problems.
\begin{itemize}
\item \textit{Highly unbalanced sample sizes}: In certain practical scenarios,
one dataset can be much bigger than the other. For example, one of the
deanonymization experiments carried out in
\cite[Section 5]{narayanan2008robust} involves the Netflix and IMDb datasets
of size $n=480{,}189$ and $m=50$, respectively. As discussed in
Remark~\ref{rmk:unbalanced}, when $m=o(n)$, the detection threshold is open
even for Gaussian models in both low and high dimensions. Letting
$\alpha =m/n\to 0$, for fixed $d$,
Corollary~\ref{cor:lower bound general} yields the impossibility condition
of $1-\rho ^{2} = \Omega (\alpha ^{1/d})$ and our general results fail
to yield any non-trivial upper bound. For growing $d$, there remains a
substantial gap between the negative result of
$\rho ^{2} \leq \frac{1}{d}\log \frac{C}{\alpha}$ (from~Corollary~\ref{cor:lower bound general})
and positive result $\rho ^{2} \gg \frac{1}{d \alpha}$ (from~Theorem~\ref{thm:upper bound general eigen r}
with $r = d$). One might hope to narrow this gap by applying Theorem
\ref{thm:upper bound general eigen r} with a larger truncation level $r$. However, in the highly unbalanced
regime of $m=o(n)$, this makes the required moment condition difficult
to verify.
\item \textit{Partially overlapped databases}: An interesting generalization
of the broken sample problem concerns partially overlapping datasets, where
two databases contain $n$ and $m$ items respectively, but only a hidden
subset of $k$ pairs is correlated. We note that upper and lower bounds
for this problem have been obtained in~\cite{elimelech2024detection} for
the case where $m = n$. The techniques proposed in the present work can
be extended to the partially overlapped case to obtain sharp thresholds
for Gaussian databases, provided that $k = \Theta (n)$. Characterizing
the sharp detection thresholds for the sparse case (i.e. $k = o(n)$) remains an open problem.
\item \textit{Adapting to the joint distribution}: The theory and methods
developed in this paper assumed that the joint distribution
${\mathsf{P}}_{X,Y}$ is known, which may not be a practical assumption.
To this end, for equal sample sizes, a promising proposal is the
\textit{Wasserstein test} (\ref{eq:wass}) discussed in
Section~\ref{sec:power analysis}, based on the intuition that the two empirical
distributions are close for correlated samples and far for independent
samples. This test is both universal (adapting to the unknown
${\mathsf{P}}_{X,Y}$), computationally appealing (scaling favorably with
the dimension), and has strong empirical performance. Identifying the distribution
class for which the Wasserstein test attains the optimal detection threshold
is an interesting open question.

In a companion paper~\cite{GongWuXu25}, we explore the challenging setting
of unknown ${\mathsf{P}}_{X,Y}$ in the context of
\textit{shuffled linear regression}, where ${\mathsf{P}}_{X,Y}$ is determined
by
$Y=\rho \left \langle X, \beta \right \rangle +\sqrt{1-\rho ^{2}} Z$ with
Gaussian covariates $X \sim {\mathcal{N}}(0,\mathbf I_{d})$, Gaussian noise
$ Z \sim {\mathcal{N}}(0,1)$, and $\beta $ being the unknown regression
coefficients uniformly distributed on the unit sphere. In this model, both
the construction and the analysis of the tests require new ideas.
\end{itemize}

\begin{appendix}
\section{Auxiliary lemmas}
\label{app:appendix}

\begin{lemma}
\label{lmm:cycle index}
Let $N_{i}$ denote the number of $i$-cycles in
$\sigma \in S_{\ell }$ for $i \in [\ell ]$. Then
\begin{equation*}
\frac{1}{\ell !} \sum _{\sigma \in S_{\ell }} \prod _{i = 1}^{\ell }
\left ( \sum _{k=0}^{\infty }\lambda _{k}^{2i} \right )^{N_{i}} = a_{
\ell },
\end{equation*}
where
$a_{\ell }= [z^{\ell }] \prod _{k=0}^{\infty }1/(1-z \lambda ^{2}_{k})$
is the $\ell $-th coefficient in the power series.
\end{lemma}
\begin{proof}
Define the cycle index polynomial of $S_{\ell }$ (see~\cite[Page 83]{stanley2018algebraic}):
\begin{equation*}
Z_{S_{\ell }} (z_{1},\cdots , z_{\ell }) = \frac{1}{\ell !} \sum _{
\sigma \in S_{\ell }} \prod _{i=1}^{\ell }z_{i}^{N_{i}}.
\end{equation*}
Then the left hand side of the desired equality is exactly the cycle index
polynomial of $S_{\ell }$ evaluated at
$(\sum _{k=0}^{\infty }\lambda _{k}^{2}, \sum _{k=0}^{\infty }
\lambda _{k}^{4},\cdots , \sum _{k=0}^{\infty }\lambda _{k}^{2\ell })$.
Additionally, define the polynomial (see~\cite[Page 87]{stanley2018algebraic})
\begin{equation*}
F_{S_{\ell }}(r_{1},r_{2},\cdots ) = \sum _{i_{1}+i_{2}+\cdots =
\ell } r_{1}^{i_{1}}r_{2}^{i_{2}} \cdots .
\end{equation*}
\par
P\'olya's Theorem~\cite[Theorem 7.7]{stanley2018algebraic} states that
\begin{equation*}
Z_{S_{\ell }}\left ( \sum _{i=1}^{\infty }r_{i},\sum _{i=1}^{\infty }r_{i}^{2},
\cdots ,\sum _{i=1}^{\infty }r_{i}^{\ell } \right ) = F_{S_{\ell }}(r_{1},r_{2},
\cdots ).
\end{equation*}
Substituting $r_{i} = \lambda _{i}^{2}$ leads to
\begin{equation*}
\frac{1}{\ell !} \sum _{\sigma \in S_{\ell }} \prod _{i = 1}^{\ell }
\left ( \sum _{k=0}^{\infty }\lambda _{k}^{2i} \right )^{N_{i}} =
\sum _{i_{1}+i_{2}+\cdots = \ell } (\lambda _{1}^{2})^{i_{1}}(
\lambda _{2}^{2})^{i_{2}}\cdots = [z^{\ell }] \prod _{k=0}^{\infty }
\frac{1}{1-z \lambda _{k}^{2}} = a_{\ell },
\end{equation*}
where the penultimate equality follows from expanding each term inside
the product as a geometric sum.
\end{proof}

\begin{lemma}
\label{lmm:limit of a_n}
If $\lambda _{1} \leq 1 - c$ and
$\sum _{k=1}^{\infty }\lambda ^{2}_{k} \leq C$ for two constants
$c\in (0,1)$ and $C > 0$, then
\begin{equation*}
\left | a_{\ell }- \prod _{k=1}^{\infty }\frac{1}{1-\lambda ^{2}_{k}}
\right | \leq \frac{C'}{r^{\ell}},
\end{equation*}
where $r > 1$ and $C' > 0$ are two constants depending on $c$ and
$C$.
\end{lemma}
\begin{proof}
Let
$f(z) \triangleq \prod _{k=0}^{\infty }\frac{1}{1-z \lambda _{k}^{2}} =
\sum _{k=0}^{\infty }a_{k} z^{k}$ and define
$g(z) = \prod _{k=1}^{\infty }\frac{1}{1-z\lambda _{k}^{2}} = (1-z)f(z)$.
Under the given conditions,
$\lambda ^{2}_{1} \le c' \equiv (1-c)^{2} < 1$. We can then fix
$r \in (1,\frac{1}{c'})$ which satisfies
$r \lambda _{1}^{2} \le rc' < 1$. Consider integral
\begin{equation*}
I = \frac{1}{2\pi \textbf{i}} \oint _{|z| = r} \frac{1}{z^{\ell +1}}f(z){
\mathrm{d}}z = \frac{1}{2\pi \textbf{i}} \oint _{|z| = r}
\frac{1}{z^{\ell +1}} \frac{1}{1-z} g(z) {\mathrm{d}}z,
\end{equation*}
where the orientation is counterclockwise. Note that the integrand
$G(z) \equiv \frac{1}{z^{\ell +1}} \frac{1}{1-z} g(z)$ has only two singularities,
$0$ and $1$, inside the contour. The singularities arising from
$g(z)$ (which occur at $1/\lambda _{k}^{2}$) lie strictly outside the contour
since $1/\lambda _{k}^{2} > r$. Applying the residue theorem yields
\begin{equation*}
I = \text{Res}(G,0)+\text{Res}(G,1) =
\frac{\left ( \frac{g(z)}{1-z} \right )^{(\ell )}\big |_{z=0}}{\ell !}
- \left ( \frac{g(z)}{z^{\ell +1}} \right )\Big|_{z=1} = a_{\ell}-g(1)
= a_{\ell}- \prod _{k=1}^{\infty }\frac{1}{1-\lambda _{k}^{2}}.
\end{equation*}
Furthermore,
\begin{equation*}
|I| \leq \frac{1}{2\pi} \oint _{|z| = r} \frac{1}{|z^{\ell +1}|}
\frac{1}{|1-z|} |g(z)| {\mathrm{d}}z \leq \frac{1}{2\pi}
\frac{1}{r^{\ell +1} (r-1)} \prod _{k=1}^{\infty }
\frac{1}{1-r\lambda _{k}^{2}} \cdot 2\pi r = \frac{1}{r^{\ell }(r-1)}
\prod _{k=1}^{\infty }\frac{1}{1-r\lambda _{k}^{2}}.
\end{equation*}
Here $r > 1$ and
\begin{equation*}
\prod _{k=1}^{\infty }\frac{1}{1-r\lambda _{k}^{2}} = \exp \left (
\sum _{k=1}^{\infty }-\ln (1-r\lambda _{k}^{2}) \right ) \leq \exp
\left ( \sum _{k=1}^{\infty }\frac{r}{1-r\lambda _{k}^{2}} \lambda _{k}^{2}
\right ) \leq \exp \left ( \frac{rC}{1-rc'} \right ),
\end{equation*}
As a result,
\begin{equation*}
\left | a_{\ell }- \prod _{k=1}^{\infty }\frac{1}{1-\lambda _{k}^{2}}
\right | \leq \frac{1}{r^{\ell }(r-1)} \exp \left ( \frac{rC}{1-rc'}
\right ) = \frac{C'}{r^{\ell}},
\end{equation*}
where $C' \triangleq \frac{\exp \left ( rC / (1-rc') \right )}{r-1}$.
\end{proof}

\begin{lemma}
\label{lmm:properties of f in proof of Bernoulli case}
Let
\[f(\rho,q) = 1-\sqrt{(1-q)^3(1-q+\rho q)} -2\sqrt{q^2(1-q)^2(1-\rho)} - \sqrt{q^3(q+\rho (1-q))},\]
where $q \equiv q(n)  \in (0,1)$ and $\rho \equiv \rho(n) \in (-\min\{\frac{1-q}{q},\frac{q}{1-q}\},1)$.
\begin{enumerate}
    \item If $q = o(1)$, then
    \[f(\rho,q) = O(q).\]
    \item If $q = o(1)$, $\rho = o(1)$ and $\frac{\abs{\rho}}{q}\diverge$, then
    \[f(\rho,q) = O(\abs{\rho} q),\]
\end{enumerate}

\end{lemma}

\begin{proof}
\begin{enumerate}
\item By Taylor expansion,
\[(1-q)^{\frac{3}{2}} = 1 - \frac{3}{2}q + o(q),\quad (1-q+\rho q)^{\frac{1}{2}} = 1 - \frac{1-\rho}{2}q + o(q),\]
implying that
\[\sqrt{(1-q)^3(1-q+\rho q)} = 1-2q + \frac{\rho q}{2} + o(q).\]

Then
\[f(\rho,q) \leq 1 - \sqrt{(1-q)^3(1-q+\rho q)} = 2q - \frac{\rho q}{2} + o(q) = O(q).\]
    \item By Taylor expansion, we have
    \[(1-q)^{\frac{3}{2}} = 1 - \frac{3}{2}q + o(q) = 1 - \frac{3}{2}q + o(\abs{\rho} q),\]
    where the last equality follows from our assumption $\frac{\abs{\rho}}{q}\diverge$ and high order terms have order at least $2$.
    
    Similarly,
    \[(1-q+\rho q)^{\frac{1}{2}} = 1 - \frac{1-\rho}{2}q +o(\abs{\rho} q),\]
hence
    \[\sqrt{(1-q)^3(1-q+\rho q)} = 1 - 2q + \frac{\rho q}{2} + o(\abs{\rho} q).\]

    By the same argument,
    \[2\sqrt{q^2(1-q)^2(1-\rho)} = 2q(1-q)\pth{1-\frac{1}{2}\rho + o(\abs{\rho})} = 2q - \rho q + o(\abs{\rho} q).\]

    Consequently,
    \[f(\rho,q) \leq 1-\sqrt{(1-q)^3(1-q+\rho q)} -2\sqrt{q^2(1-q)^2(1-\rho)} = \frac{\rho q}{2} + o(\abs{\rho} q) = O(\abs{\rho} q),\]
    where $\rho$ must be positive otherwise condition $\abs{\rho} / q \to \infty$ cannot hold.

\end{enumerate}
\end{proof}

\begin{lemma}
\label{lmm:property of f-divergence}

Let $\sfP,\sfQ$ be two probability measures on $\calX$ with a $\sigma$-algebra $\calF$. Given a finite $\calF$-measurable partition $\calE = \{E_1,E_2,\cdots,E_n\}$, define the distribution $\sfP_\calE$ on $[n]$ by $\sfP_\calE(i) = \sfP(E_i)$ and $\sfQ_\calE(i) = \sfQ(E_i)$. Then we have
\[\chi^2(\sfP\|\sfQ) = \sup_{\calE} \chi^2(\sfP_\calE\|\sfQ_\calE),\]
where the supremum is over all finite $\calF$-measurable partitions $\calE$.
\end{lemma}
\begin{proof}
    See~\cite[Theorem 7.6]{polyanskiy2025information}.
\end{proof}

\begin{lemma}
\label{lmm:property of f-divergence on product space}
Let ${\mathsf{P}},{\mathsf{Q}}$ be two probability measures on
${\mathcal{X}}\times {\mathcal{Y}}$ with $\sigma $-algebra
${\mathcal{F}}_{{\mathcal{X}}}\otimes {\mathcal{F}}_{{\mathcal{Y}}}$. Given
a finite ${\mathcal{F}}_{{\mathcal{X}}}$-measurable partition
${\mathcal{I}}= \{I_{1},I_{2},\cdots ,I_{w}\}$ of ${\mathcal{X}}$ and a
finite ${\mathcal{F}}_{{\mathcal{Y}}}$-measurable partition
${\mathcal{J}}= \{J_{1},J_{2},\cdots ,J_{w}\}$ of ${\mathcal{Y}}$,\footnote{Without
loss of generality, we assume the partitions ${\mathcal{I}}$ and
${\mathcal{J}}$ have equal sizes} their induced partition on
${\mathcal{X}}\times {\mathcal{Y}}$ is defined as
${\mathcal{E}}= \{I_{k} \times J_{l}\}_{k,l = 1}^{w}$. Define the distribution
${\mathsf{P}}_{{\mathcal{E}}}$ on $\{(k,l)\}_{k,l=1}^{w}$ by
${\mathsf{P}}_{{\mathcal{E}}}(k,l) = {\mathsf{P}}(I_{k}\times J_{l})$ and
${\mathsf{Q}}_{{\mathcal{E}}}(k,l) = {\mathsf{Q}}(I_{k}\times J_{l})$.
Then we have
\begin{equation*}
\chi ^{2}({\mathsf{P}}\|{\mathsf{Q}}) = \sup _{{\mathcal{E}}}\chi ^{2}({
\mathsf{P}}_{{\mathcal{E}}}\|{\mathsf{Q}}_{{\mathcal{E}}}),
\end{equation*}
where the supremum is over all possible partitions ${\mathcal{E}}$ induced
by finite ${\mathcal{F}}_{{\mathcal{X}}}$-measurable partition
${\mathcal{I}}$ and finite ${\mathcal{F}}_{{\mathcal{Y}}}$-measurable partition
${\mathcal{J}}$.
\end{lemma}

\begin{proof}
By Lemma~\ref{lmm:property of f-divergence}, $\chi^2(\sfP\|\sfQ) = \sup_{\calG} \chi^2(\sfP_\calG\|\sfQ_\calG)$, where
the supremum is over all finite $\calF_{\calX} \otimes \calF_{\calY}$-measurable partition $\calG$. Thus it suffices to show that we can restrict the supremum to be over all rectangles that are Cartesian product of $\calF_\calX$-measurable partition $\calI$ and finite $\calF_\calY$-measurable partition $\calJ$.
This is known for mutual information (see, e.g., \cite[(4.29)]{polyanskiy2025information}) and we next show it also holds for $\chi^2$-information.

To this end, fix a partition $\calG = \{G_1,G_2,\cdots,G_n\}$. Let $\sfP_\calG(i) = \sfP(G_i)$ and $\sfQ_\calG(i) = \sfQ(G_i)$. Fixing a constant $\epsilon > 0$, we construct a partition $\calE$ induced by partitions on $\calX$ and $\calY$ such that $\chi^2(\sfP_\calE \| \sfQ_\calE) \geq \chi^2(\sfP_\calG\|\sfQ_\calG) - \epsilon$.

Fixing any $\delta > 0$ and $a = \frac{1}{2^n}$, one can find $\tG_1$ such that $(\sfP+\sfQ)(G_1\Delta \tG_1) < a\delta$, where $\tG_1$ is a finite union of measurable rectangles (see~\cite[Exercise 1.4.28]{tao2011introduction}), and $\Delta$ denotes symmetric difference. Similarly, one can find $G_2^*$ which is also a finite union of measurable rectangles such that $(\sfP+\sfQ)(G_2^*\Delta G_2) < a\delta$. Then define $\tG_2 =  G_2^* \setminus \tG_1$ and observe that
    \begin{align*}
        G_2\Delta \tG_2  = (\tG_2\setminus G_2) \cup (G_2\setminus \tG_2)
        \subset (G_2^* \setminus G_2) \cup ((G_2\setminus G_2^*) \cup (G_2 \cap \tG_1))
         \subset (G_2^* \Delta G_2 )\cup (G_2 \cap \tG_1).
    \end{align*}
    Since $G_1 \cap G_2 = \emptyset$, we have
    $G_2 \cap\tG_1 \subset \tG_1 \Delta G_1$. Consequently, $G_2 \Delta \tG_2 \subset (G_2^* \Delta G_2) \cup (\tG_1 \Delta G_1)$, and
    hence $(\sfP+\sfQ)(G_2 \Delta \tG_2) \leq a\delta + a\delta = 2a\delta$.

    Repeat the above process and obtain $\tG_1,\tG_2,\cdots,\tG_{n-1}$ such that they are mutually disjoint and satisfy $(\sfP+\sfQ)(G_{k} \Delta \tG_k) \leq 2^{k-1}a\delta$ for $k \in [n-1]$. Finally define $\tG_n = (\calX\times \calY) \setminus (\tG_1\cup \tG_2\cdots \cup \tG_{n-1})$. One can check that $G_n \Delta \tG_n \subset (G_1\Delta \tG_1) \cup \cdots \cup (G_{n-1}\Delta \tG_{n-1})$
    and $(\sfP+\sfQ)(G_n\Delta \tG_n) \leq (2^{n}-1)a \delta < \delta$.

    Note that the difference between two finite unions of measurable rectangles is still a finite union of measurable rectangles. Therefore, $\widetilde{\calG} = \{\tG_1,\tG_2,\cdots,\tG_n\}$ is a partition of $\calX\times \calY$ and each of them is a finite union of measurable rectangles satisfying $(\sfP+\sfQ)(\tG_k \Delta G_k) < \delta$ for $k \in [n]$. Denote $\tG_k = \cup_{i=1}^{m_k} A_i^{(k)}\times B_i^{(k)}$ for $k \in [n]$, where $A_i^{(k)} \in \calF_\calX$, $B_i^{(k)} \in \calF_\calY$, and $A_i^{(k)}\times B_i^{(k)}$ are disjoint across different $i$.

For each $k \in [n]$,  there exists $\delta_k$ such that if $\delta< \delta_k$, then
$$
\frac{\sfP^2(G_k)}{\sfQ(G_k)} \leq 
\frac{(\sfP(G_k)-\delta)^2}{\sfQ(G_k)+\delta} + \frac{\epsilon}{n}
\leq \frac{\sfP^2(\tG_k)}{\sfQ(\tG_k )} +\frac{\epsilon}{n}.
$$

Thus if $\delta < \min_{k \in [n]}\delta_k,$
then 
    \[\chi^2(\sfP_\calG\|\sfQ_\calG) = \sum_{k=1}^n \frac{\sfP^2(G_k)}{\sfQ(G_k)} - 1 \leq \epsilon + \sum_{k=1}^n \frac{\sfP^2(\tG_k)}{\sfQ(\tG_k)}-1 = \epsilon+ \chi^2(\sfP_{\widetilde{\calG}}\| \sfQ_{\widetilde{\calG}}).\]
Let $\calM = \sth{A_i^{(k)}\times B_i^{(k)}: 1\leq i \leq m_k,1\leq k\leq n}$ denote the partition on $\calX\times \calY$ with only measurable rectangles. Consider a channel which maps $(i,k)$ to $k$. If the input follows distribution $\sfP_{\calM}$ (resp.\ $\sfQ_{\calM}$), then the output follows distribution $\sfP_{\widetilde{\calG}}$ (resp.\ $\sfQ_{\widetilde{\calG}}$). By the data processing inequality for chi-squared divergence (see e.g.~\cite[Theorem 7.4]{polyanskiy2025information}), we have
\begin{align}
\chi^2(\sfP_{\widetilde{\calG}}\| \sfQ_{\widetilde{\calG}}) \le
\chi^2(\sfP_\calM\|\sfQ_\calM). \label{eq:data_processing}
\end{align}
Combining it with the last displayed equation yields that 
$$\chi^2(\sfP_\calG\|\sfQ_\calG) \le \chi^2(\sfP_\calM\|\sfQ_\calM)
+ \epsilon.$$

    Note that $\calM$ may not be induced by partitions on $\calX$ and $\calY$. Thus, we define 
    \[\calI = \sth{\cap_{k=1}^n \cap_{i=1}^{m_k} C_i^{(k)}: C_i^{(k)} \in \sth{A_i^{(k)},\pth{A_i^{(k)}}^{\rm c}}},\]
and relabel elements of $\calI$ to be $\{I_1,I_2,\cdots,I_w\}$. Sets in $\calI$ are disjoint and $\cup_{k=1}^w I_k =\calX$;  hence $\calI$ is a partition of $\calX$. Similarly, define
\[\calJ =\sth{\cap_{k=1}^n \cap_{i=1}^{m_k} D_i^{(k)}: D_i^{(k)} \in \sth{B_i^{(k)},\pth{B_i^{(k)}}^{\rm c}}},\]
and relabel elements of $\calJ$ to be $\{J_1,J_2,\cdots,J_w\}$, so that $\calJ$ is a partition of $\calY$. Let $\calE = \{I_k\times J_l\}_{k,l=1}^{w}$.
Note that for any $A_i^{(k)} \times B_i^{(k)}\in \calM$, $A_i^{(k)}$ must be the union of some disjoint elements in $\calI$, and so is $B_i^{(k)}$. Then $A_i^{(k)}\times B_i^{(k)}$ is a union of disjoint elements in $\calE$. Therefore, analogous to~\prettyref{eq:data_processing}, we can show that $$
\chi^2(\sfP_\calM\|\sfQ_\calM)\leq\chi^2(\sfP_\calE\|\sfQ_\calE).$$ 
It follows that 
    \[
    \chi^2(\sfP_\calG\|\sfQ_\calG)
    \le \chi^2(\sfP_\calM\|\sfQ_\calM)+\epsilon
    \le
\chi^2(\sfP_\calE\|\sfQ_\calE) + \epsilon.\]
    Taking the supremum on both sides yields
    \[
\sup_\calG\chi^2(\sfP_\calG\|\sfQ_\calG)
    \le 
    \sup_\calE \chi^2(\sfP_\calE\|\sfQ_\calE)+ \epsilon.\]
    Since $\epsilon$ is arbitrary, we conclude that $\sup_\calG\chi^2(\sfP_\calG\|\sfQ_\calG) 
    \le \sup_\calE \chi^2(\sfP_\calE\|\sfQ_\calE)$. Also, by definition, it trivially holds that 
$\sup_\calG\chi^2(\sfP_\calG\|\sfQ_\calG) 
    \ge \sup_\calE \chi^2(\sfP_\calE\|\sfQ_\calE)$.
Thus, by~\prettyref{lmm:property of f-divergence},
we finally deduce that 
$\chi^2(\sfP\|\sfQ) = \sup_\calG\chi^2(\sfP_\calG\|\sfQ_\calG) 
    = \sup_\calE \chi^2(\sfP_\calE\|\sfQ_\calE)$.
\end{proof}

\section{Proof of Corollaries~\ref{cor:Gaussian} and~\ref{cor:Bernoulli}}

\begin{proof}[Proof of~\prettyref{cor:Gaussian}]
Let $\sfp = \calN\pth{\zeros,\begin{bmatrix}
  1&\rho \\
 \rho &1
\end{bmatrix}}$ and $\sfq = \calN\pth{\zeros,\begin{bmatrix}
  1&0 \\
 0 &1
\end{bmatrix}}$. Then $\sfP_{X,Y} = \sfp^{\otimes d}$ and $\sfP_X\otimes\sfP_Y = \sfq^{\otimes d}$. The likelihood ratio defined in~\prettyref{eq:likelihood ratio} is $L(x,y) = \prod_{j=1}^d \ell(x_j,y_j)$ for $x,y\in \reals^d$,
where 
$$\ell(a,b) = \frac{\sfp(a,b)}{\sfq(a,b)} = \frac{1}{\sqrt{1-\rho^2}}\rexp{\frac{-(a^2+b^2)\rho^2 + 2ab\rho}{2(1-\rho^2)}}, \quad \forall a,b \in \reals
$$
is the Mehler kernel, which is diagonalized by Hermite polynomials (see~\cite[Equation 4.16]{janson1997gaussian}):
\begin{equation}
\ell(a,b) = \sum_{k=0}^\infty \rho^k \frac{H_k(a)}{\sqrt{k!}}\frac{H_k(b)}{\sqrt{k!}},
    \label{eq:mehler}
\end{equation}
where 
$\{H_k\}_{k\geq0}$ are probabilists' Hermite polynomials defined as $H_n(x)=(-1)^n \rexp{x^2/2}\frac{\diff^n }{\diff x^n}\rexp{-x^2/2}$ and $\{\rho^k\}_{k\geq0}$ are eigenvalues of $\ell$. From here, one can obtain the decomposition of $L$ whose eigenfunctions are multivariate Hermite polynomials of the form $\prod_{j=1}^d \frac{1}{\sqrt{k_j!}} H_{k_j}(x_j)$. In particular, the first $d$ eigenfunctions are linear:
for $k \in [d]$, 
$\lambda_k = \rho$, $\phi_k(x) = x_k$ and $\psi_k(y) = y_k$. Then
\[\Ichi(X;Y)= \chi^2(\sfP_{X,Y}\|\sfP_X\otimes \sfP_Y) = \Expect_{\sfP_X\otimes \sfP_Y} L^2 - 1 = \pth{\Expect_{\sfq}\ell^2(x_1,y_1)}^d-1 =\frac{1}{(1-\rho^2)^d}-1,\]
where the third equality holds because entries of $X$ and $Y$ are independent across $d$ dimensions and the last equality follows from the identity $\Expect_{\sfq} \ell^2(x_1,y_1) = \sum_{k=0}^\infty \rho^{2k} = \frac{1}{1-\rho^2}$.

We now prove the theorem for fixed $d$ and growing $d$ separately.

\begin{itemize}
    \item Case 1: fixed $d$. 

We show that $\rho^2 \to 1$ is sufficient and necessary for strong detection. For sufficiency, suppose $d = 1$ or only using the first coordinate of $X_i$'s and $Y_i$'s for $d >1$, and consider $\Phi_r(X)$ and $\Psi_r(Y)$ defined in~\prettyref{eq:eigenvector test}. For any fixed $\epsilon>0,$ choose 
\begin{equation}
r=\frac{16}{\alpha} \log \frac{1}{\epsilon}.
\label{eq:r-choice}
\end{equation}
 By assumption, $\rho^2 \equiv \rho(n)^2 \to 1$. Thus, there exists $n_\epsilon$ such that for all $n \ge n_\epsilon,$ $1-\rho^2 \le 1/(r+1)$. 
Since we can always add independent noise to both samples, we can assume without loss of generality that $1-\rho^2 =  1/(r+1)$ for all $n \ge n_\epsilon$.

By the CLT, $(\Phi_r, \Psi_r)$ converges in distribution to $\sfN_0 = \calN(\zeros_{2r},\identity_{2r})$ under $\calH_0$ and $\sfN_1 = \calN(\zeros_{2r},\Sigma_{2r})$ under $\calH_1$ as $n\diverge$, where
\[\Sigma_{2r} = \begin{bmatrix}
 \identity_{r} & \sqrt{\alpha} \Lambda_r\\
 \sqrt{\alpha} \Lambda_r &\identity_{r}
\end{bmatrix},\]
with $\Lambda_r = \diag{\rho,\rho^2,\cdots,\rho^r}$. Moreover, following the same argument as in~\prettyref{eq:bound_total_variation_VTEX1}, we get that
\[1 - \dTV(\sfN_0,\sfN_1)
\le 1- H^2(\sfN_0, \sfN_1)
= \prod_{k=1}^r \left(\frac{1-\alpha \rho^{2k} }{(1-\alpha \rho^{2k}/4)^2}\right)^{1/4}
\le \exp \left( - \frac{\alpha}{8}\sum_{k=1}^r \rho^{2k}\right).
\]
Now, plugging in $\rho^2 =1-1/(r+1)$, we get that 
$$
\sum_{k=1}^r \rho^{2k} = \sum_{k=1}^r \left( 1- \frac{1}{r+1}\right)^k \ge 
\sum_{k=1}^r 
\left( 1- \frac{k}{r+1} \right) =\frac{r}{2}.
$$
Therefore, 
\[1 - \dTV(\sfN_0,\sfN_1)  \leq \exp\pth{-\frac{\alpha}{16}r} = \epsilon.\]

Define $S$ to be the likelihood ratio test statistic between $\sfN_0$ and $\sfN_1$, \ie
\[S = -\frac{1}{2}Z^\top \pth{\Sigma_{2r}^{-1} - \identity_{2r}}Z-\frac{1}{2}\sum_{k=1}^r (1-\alpha \rho^{2k}),\]
where $Z \sim \sfN_0$ under $\calH_0$
and $Z \sim \sfN_1$ under $\calH_1. $
It follows that $\Prob_0 \sth{S > 0} + \Prob_1 \sth{S \leq 0} \leq \epsilon$.
Define $T_n$ as 
\[T_n = -\frac{1}{2} \begin{pmatrix}
\Phi_r(X) \\\Psi_r(Y)
\end{pmatrix}^\top \pth{\Sigma_{2r}^{-1} - \identity_{2r}}\begin{pmatrix}
\Phi_r(X) \\\Psi_r(Y)
\end{pmatrix} - \frac{1}{2}\sum_{k=1}^r(1-\alpha \rho^{2k}).\]
Since $T_n$ converges to $S$ in distribution under both $\calH_0$ and $\calH_1,$
 $\lim_{n\diverge} \Prob_0 \sth{T_n> 0} + \Prob_1 \sth{T_n \leq 0} \leq \epsilon$. Finally, since $\epsilon$ is arbitrary, we get $\lim_{n\diverge} \Prob_0 \sth{T_n> 0} + \Prob_1 \sth{T_n \leq 0} =0$.

For necessity, assume that $\rho^2 \not \to 1$, \ie, it is bounded away from $1$. Then we have
\[\Ichi(X;Y)  = \frac{1}{(1-\rho^2)^d} - 1 = O(1)\text{ and } \rho(X;Y) = \rho = 1-\Omega(1),\]
which implies that strong detection is impossible according to~\prettyref{cor:lower bound general}.

\item Case 2: growing $d$. The goal is to show that strong detection is possible if and only if $\rho^2d \to \infty$. For sufficiency, let $r = d$, then we have
$\frac{m}{n}\Ichir(X;Y) = \frac{m}{n}\sum_{k=1}^d \lambda^2_k = \frac{m}{n}\rho^2d \to \infty$ and $L_r(x,y)=\sum_{k=1}^d \rho \phi_k(x)\psi_k(y)=\rho \iprod{x}{y}. $
Therefore, 
\[\Expect_{\sfP_{X,Y}} \qth{L^2_r(X,Y)} 
=\rho^2 \Expect_{\sfP_{X,Y}} \iprod{X}{Y}^2 = \rho^4(2d + d^2) + (1-\rho^2)\rho^2d.\]

Consequently, 
\[\frac{\Expect_{\sfP_{X,Y}} \qth{L^2_r(X,Y)}}{m\pth{\Ichi^{(r)}(X;Y)}^2}=\frac{\rho^2 (2d+d^2) + (1-\rho^2)d}{m\rho^2 d^2} \to 0,\]
implying that strong detection is achievable by~\prettyref{thm:upper bound general eigen r}.

To show the necessity, assume $\rho^2d \not \diverge$, then it is bounded and we must have $\rho \to 0$ and then $\rho(X;Y)=\lambda_1 = \rho = 1-\Omega(1)$. Additionally, for $n$ large enough, $1 - \rho^2 \geq \frac{1}{2}$ and then
\[\Ichi(X;Y)=\sum_{k=1}^\infty \lambda_k^2 = \frac{1}{(1-\rho^2)^d} -1 \leq  \rexp{-d\ln (1-\rho^2)} \leq \rexp{d \frac{\rho^2}{1-\rho^2}} \leq \rexp{2\rho^2 d} = O(1),\]
where the second inequality follows from $-\ln(1-x) \leq \frac{x}{1-x}$. By~\prettyref{cor:lower bound general}, we know that strong detection is impossible.
\end{itemize}
\end{proof}

\begin{proof}[Proof of~\prettyref{cor:Bernoulli}]

In~\prettyref{tab:Bernoulli pmf}, we summarize the probability mass functions of $\sfq$ and $\sfp$.

\begin{table}[!thp]
\centering
\begin{tabular}{|l|l|l|l|l|}
\hline 
& $(0,0)$ & $(1,0)$ & $(0,1)$ & $(1,1)$
\\
\hline 
$\sfq$ & $(1-q)^2$ & $q(1-q)$ & $q(1-q)$ & $q^2$\\
\hline
$\sfp$& $(1-q)(1-q+\rho q)$ & $q(1-q)(1-\rho)$ & $q(1-q)(1-\rho)$ & $q(q+\rho(1-q))$\\
\hline
\end{tabular}
\caption{Joint probability mass functions for the Bernoulli model.}
\label{tab:Bernoulli pmf}
\end{table}

The likelihood ratio defined in~\prettyref{eq:likelihood ratio} is
$L(x,y) = \prod_{j=1}^d \ell(x_j,y_j) $ for $x,y \in \{0,1\}^d,$
where $
\ell(a,b) = \frac{\sfp(a,b)}{\sfq(a,b)}$ for $a,b \in \{0,1\}$. Let $g(a) = \frac{a-q}{\sqrt{q(1-q)}}$. Then $\ell(a,b)$ has the orthogonal decomposition as $\ell(a,b) = 1 + \rho g(a)g(b)$. Thus, we obtain that $\Ichi(X;Y) = (1+\rho^2)^d - 1$. 
Additionally, we derive the corresponding decomposition for $L$ and obtain that for $k \in [d]$, $\lambda_k = \rho$, $\phi_k(x) = g(x_k)$ and $\psi_k(y) = g(y_k)$.

We consider fixed $d$ and growing $d$ separately.

\begin{itemize}
    \item Case 1: fixed $d$. 
    \begin{itemize}
        \item $\alpha \in (0,1)$: In this case, we always have $\sum_{k=1}^\infty \lambda^2_k = (\rho^2+1)^d - 1 = O(1),$
        meaning that strong detection is impossible by~\prettyref{cor:lower bound general}.
        \item $\alpha = 1$:
The goal is to show that strong detection is possible if and only if $\rho^2 \to 1$ and $nq\diverge$. For sufficiency, we only need to verify $\Expect_{\sfP_X} |\phi_1^3(X)| = o(\sqrt{n})$. In fact, 
\begin{equation}
    \label{eq:third moment condition}\frac{\Expect_{\sfP_X}|\phi_1^3(X)|}{\sqrt{n}} = \frac{\Expect_{x \in \Bern(q)} |x-q|^3}{\sqrt{n}q^{3/2}(1-q)^{3/2}} = \frac{(1-q)^{3/2}}{\sqrt{qn}} + \frac{q^{3/2}}{\sqrt{n(1-q)}}\to 0,
\end{equation}
given $nq \diverge$ and $q \le 1/2$ by assumption. Hence, strong detection is possible according to~\prettyref{thm:upper bound general eigen 1}.

For necessity, if $\rho^2 \not \to 1$, then there exists a subsequence along which $\rho$ is bounded away from $1$. Then
\[\Ichi(X;Y) = (\rho^2 + 1)^d - 1 = O(1) 
 \text{ and } \rho(X;Y) = \rho = 1 -\Omega(1),\]
 implying that strong detection is not achievable by applying~\prettyref{cor:lower bound general}.

To show $nq \diverge$ is also necessary, note that if the latent injection $\pi$ is known, in which case the problem becomes testing $\sfP_{X,Y}$ vs $\sfP_X\otimes \sfP_Y$ based on $m$ independent observations. It is well-known that strong detection for this easier problem is equivalent to $H^2(\sfP_{X,Y},\sfP_X\otimes \sfP_Y)\gg 1/m$. We have 
\[mH^2(\sfP_{X,Y},\sfP_X\otimes \sfP_Y) = mH^2(\sfp^{\otimes d}, \sfq^{\otimes d}) = 2m \pth{1 - \pth{1- \frac{H^2(\sfp,\sfq)}{2}}^d},\]
meaning that if $mdH^2(\sfp,\sfq) = O(1)$, then $mH^2(\sfP_{X,Y},\sfP_X\otimes \sfP_Y) = O(1)$, \ie strong detection is impossible (even if the latent injection is known).

If $nq \not \diverge$, we verify that $mdH^2(\sfp,\sfq) = O(1)$. According to~\prettyref{tab:Bernoulli pmf}, we have
\begin{align*}
    {H^2(\sfp,\sfq)} &= 2-2\sum_{i,j} \sqrt{\sfq(i,j)\sfp(i,j)}\\
    & = 2 - 2\sqrt{(1-q)^3 (1-q+\rho q)} - 4\sqrt{q^2(1-q)^2 (1-\rho)} - 2\sqrt{q^3 (q+\rho(1-q))}.
\end{align*}
Given assumption $nq\not \to \infty$, we have $q = o(1)$. Then by~\prettyref{lmm:properties of f in proof of Bernoulli case}, ${H^2}(\sfp,\sfq) = O(q)$, implying that $mdH^2(\sfp,\sfq) = ndH^2(\sfp,\sfq)= O(ndq) = O(1)$.
    \end{itemize}
\item Case 2: growing $d$. We show that strong detection is possible if and only if $\rho^2 d \diverge$ and $nd q\abs{\rho}\diverge$. For sufficiency, choose $r = d$ in~\prettyref{thm:upper bound general eigen r}. Then $\Ichir(X;Y) = \sum_{k=1}^d \lambda_k^2 = \rho^2 d \to \infty$ and $L_r(x,y) = \sum_{k=1}^d \rho \phi_k(x)\psi_k(y)= \sum_{k=1}^d (x_k-q)(y_k-q)/[q(1-q)]$. Note that
\begin{align*}
\Expect_{\sfP_{X,Y}}\qth{\pth{\sum_{k=1}^d (x_k-q)(y_k-q)}^2} & =\Expect_{\sfP_{X,Y}} \qth{\sum_{k=1}^d (x_k-q)^2(y_k-q)^2}\\
    & +\Expect_{\sfP_{X,Y}} \qth{\sum_{k\neq k'} (x_k-q)(y_k-q)(x_{k'}-q)(y_{k'}-q)}.
\end{align*}
The first term in the above display is
\begin{align*}
    & \Expect_{\sfP_{X,Y}} \qth{\sum_{k=1}^d (x_k-q)^2(y_k-q)^2} \\
    & = d\Expect_{(x,y)\sim \sfp}[(x-q)^2(y-q)^2] \\
    & = d(1-q)(1-q+\rho q)q^4 + 2dq^3(1-q)^3(1-\rho) + d(1-q)^4q(q+\rho(1-q)).
\end{align*}
The second term is
\begin{align*}
    \Expect_{\sfP_{X,Y}} \qth{\sum_{k\neq k'} (x_k-q)(y_k-q)(x_{k'}-q)(y_{k'}-q)} & = \sum_{k\neq k'} \Expect_\sfp^2\qth{ (x-q)(y-q)} \\
    & = (d^2-d)\rho^2q^2(1-q)^2.
\end{align*}

Combining the above results, one has
\begin{align*}
&\frac{\Expect_{\sfP_{X,Y}} \qth{L^2_r(X,Y)}}{m \pth{\Ichir(X;Y)}^2}  = \frac{\Expect_{\sfP_{X,Y}}\qth{\pth{\sum_{k=1}^d (x_k-q)(y_k-q)}^2}}{\alpha n\rho^2d^2q^2(1-q)^2}\\
    & = \frac{(1-q)^4q(q+\rho(1-q)) + q^4(1-q)(1-q+\rho q) + 2q^3(1-q)^3(1-\rho)}{\alpha n\rho^2 dq^2(1-q)^2} + \frac{d-1}{\alpha nd}\\
    & = \frac{(1-q)^2 (q+\rho(1-q))}{\alpha n\rho^2d q} + \frac{q^2(1-q+\rho q)}{\alpha n\rho^2d(1-q)} + \frac{2q(1-q)(1-\rho)}{\alpha n\rho^2 d}  + \frac{d-1}{\alpha nd}.
\end{align*}

By condition $\rho^2 d \diverge$, the last three terms are vanishing. Then we require the first term tends to $0$, which is essentially
$\frac{q+\rho - \rho q}{n\rho^2 d q} \to 0$.
Note that
\[\frac{q+\rho - \rho q}{n\rho^2 d q} = \frac{1}{n\rho^2 d} + \frac{1-q}{nd\rho q} \to 0,\]
which means that the conditions of~\prettyref{thm:upper bound general eigen r} hold given $\rho^2d \diverge$ and $nd q\abs{\rho}\diverge$. Therefore, strong detection is achievable.

Now, we prove the necessity which is to show that if  $\rho^2 d \not \to \infty$ or $nd q\abs{\rho} \not\to \infty$, then strong detection is impossible. 
If $\rho^2 d \not \diverge$, then we have $\rho \to 0$ hence $\rho(X;Y) =  \rho = 1-\Omega(1)$. Additionally, $\Ichi(X;Y) =  (\rho^2 + 1)^d - 1 \leq e^{\rho^2 d} = O(1),$
meaning that strong detection is impossible by~\prettyref{cor:lower bound general}.

If $\rho^2d \diverge$ and $nd q \abs{\rho}\not \diverge$, then $nd q\abs{\rho} = O(1)$. Same as the fixed dimension scenario, we want to show that $mdH^2(\sfp,\sfq) =O(1)$ holds hence strong detection is impossible.

There are two subcases.

\begin{itemize}
   \item If $\abs{\rho} \not \to 0$, then $q \to 0$. According to~\prettyref{lmm:properties of f in proof of Bernoulli case}, we have $md{H^2(\sfp,\sfq)} = nd{H^2(\sfp,\sfq)}= O(ndq) = O(nd q\abs{\rho}) = O(1)$.
    \item If $\abs{\rho} \to 0$, given $\rho^2d \diverge$ and $nd q\abs{\rho}\not \diverge$, we have $\abs{\rho}/q \diverge$ and $q\to 0$. Then by~\prettyref{lmm:properties of f in proof of Bernoulli case}, we know that  $md{H^2(\sfp,\sfq)} = nd{H^2(\sfp,\sfq)} = O(nd q\abs{\rho}) = O(1)$. 
\end{itemize}

\end{itemize}
\end{proof}
\end{appendix}

\paragraph*{Acknowledgments}
Part of the work was supported by the National Science Foundation under
Grant No. DMS-1928930, while Y.~Wu was in residence at the Simons Laufer
Mathematical Sciences Institute in Berkeley, California, during the Spring
2025 semester. J.~Xu is supported in part by an NSF CAREER award CCF-2144593.

\bibliographystyle{abbrv}
\bibliography{references}

\end{document}